\documentclass[a4paper,11pt, reqno]{amsart}
\usepackage{amsthm, amsmath, amssymb, amsfonts, amsopn, color}
\usepackage{enumerate}

\newtheorem{theorem}{Theorem}[section]

\newtheorem{proposition}[theorem]{Proposition}
\newtheorem{lemma}[theorem]{Lemma}
\numberwithin{equation}{section}

  \newcommand{\e}{\varepsilon}
\newcommand{\RN}{\mathbb{R}^2}
\newcommand{\LN}{L^2(\mathbb{R}^2)}

\begin{document}
\title[Existence of mixed type solutions]{Existence of  mixed type
solutions in the Chern-Simons gauge theory of rank two in $\mathbb{R}^2$}
\author{Kwangseok Choe}
\address[Kwangseok Choe]{Department of Mathematics, Inha University,
Incheon, 402-751, Korea}
\email{kschoe@inha.ac.kr}
\author{Namkwon Kim}
\address[Namkwon Kim]{Department of Mathematics, Chosun University, Kwangju 501-759, Republic of Korea.}
\email{kimnamkw@chosun.ac.kr}
\author{Youngae Lee}
\address[Youngae Lee]{National Institute for Mathematical Sciences,
KT Daeduk 2 Research Center, 70 Yuseong-daero 1689 beon-gil, Yuseong-gu,
Daejeon, 34047, Republic of Korea}
\email{youngaelee0531@gmail.com}
\author{Chang-Shou Lin}
\address[Chang-Shou Lin]{Taida Institute for Mathematical Sciences, Center
for Advanced Study in Theoretical Sciences, National Taiwan University,
No.1, Sec. 4, Roosevelt Road, Taipei 106, Taiwan}
\email{cslin@math.ntu.edu.tw }

\begin{abstract}
We consider the Chern-Simons gauge theory of rank $2$
such as $SU(3)$, $SO(5)$, and $G_2$ Chern-Simons model in $\RN$.
There may exist  three types of solutions in these theories, that is, topological, nontopological,
and mixed type solutions.
Among others, mixed type solutions can only exist in non-Abelian Chern-Simons models.
We show the existence of  mixed type solutions with an arbitrary configuration of vortex points which
 has been a long-standing open problem.
To show it, as the first step, we need to find when a priori bound would fail.
For the purpose, we shall find partially blowing up mixed type solutions by using different scalings for different components.  Due to the different scalings, we should control the mass contribution from infinity which is one of the important parts in this paper.
\end{abstract}

\date{\today }
\maketitle


\section{Introduction}

In this article, we are interested in the non-Abelian relativistic self-dual Chern-Simons models
proposed by Kao-Lee \cite{KL} and Dunne \cite{D1,D2,D3}.
These models are defined in the $2+1$ Minkowski space $\mathbb{R}^{1,2}$ with metric tensor
$g_{\mu\nu}=\textrm{diag}(-1,1,1)$.  The corresponding gauge groups are compact Lie groups with semi-simple Lie algebras
$\mathcal{G}$ and Lie bracket $[\cdot,\cdot]$ over $\mathcal{G}$.
In the adjoint representation, the Lagrangian density is given by
\[ \mathcal{L}=-{\rm tr} \Big( (D_\mu\phi)^\dagger D^\mu\phi\Big) -\kappa\epsilon^{\mu\nu\rho}
 {\rm tr} \Big( \partial_\mu A_\nu A_\rho +\frac{2}{3}A_\mu A_\nu A_\rho\Big) -V(\phi,\phi^\dagger), \]
where the gauge-invariant scalar field potential $V(\phi,\phi^\dagger)$ is defined by
\[ V(\phi,\phi^\dagger)=\frac{1}{4\kappa^2} {\rm tr} \Big(([[\phi,\phi^\dagger],\phi]
 -v^2\phi)^\dagger ([[\phi,\phi^\dagger],\phi] -v^2\phi)\Big). \]
Here $D_\mu=\partial_\mu+[A_\mu,\cdot]$ is the covariant derivative, ${\rm tr}$ refers to the trace
in a finite dimensional representation of the compact semi-simple Lie group $\mathcal{G}$
to which the gauge fields $A_\mu$ and the charged scalar matter fields $\phi$ and $\phi^\dagger$ belong.
The parameter $v^2>0$ is the symmetry breaking parameter, $\epsilon^{\mu\nu\rho}$ is the Levi-Civita antisymmetric tensor
with $\epsilon^{012}=1$, and $\kappa>0$ is the Chern-Simons coupling parameter.
In the static situation, by the Bogomolinyi reduction argument,  one can obtain the self-dual equations
of the above Lagrangian :
\begin{equation}
 \left\{ \begin{array}{l}
 D_{-}\phi  =0,  \label{cs2} \\
 F_{+-}=\frac{1}{\kappa^{2}}\left[ v^2\phi -\left[ \left[ \phi ,\phi ^{\dagger}%
 \right] ,\phi \right] ,\phi ^{\dagger}\right] ,
\end{array} \right.
\end{equation}
where $D_{-}=D_{1}-{\rm i}D_{2}$ and $F_{+-}=\partial_{+}A_{-}-\partial _{-}A_{+}+\left[ A_{+},A_{-}\right]$
with $A_{\pm}=A_{1}\pm {\rm i}A_{2}$ and $\partial_{\pm} =\partial_{1}\pm {\rm i}\partial_{2}$.
It is well known that a solution of the self-dual equations is automatically a critical point of the Lagrangian.
Dunne considered a simplified form of the self-dual system \eqref{cs2} by an Ansatz, in which the fields
$\phi$ and $A$ are algebraically restricted:%
\begin{equation*}
 \phi =\sum_{a=1}^{r}\phi^{a}E_{a}, \qquad A_\mu= {\rm i}\sum_{a=1}^r A^a_{\mu}H_a,
\end{equation*}%
where $r$ is the rank of the gauge Lie algebra, $E_{a}$ is a simple root step operator,
$H_a$ is a Cartan subalgebra element,  $\phi^{a}$ is a complex valued function,
and $A^a_{\mu}$ is a real valued function. Let further%
\begin{equation*}
 u_{a}=\ln |\phi^{a}|^2 -\ln v^2, ~\quad a=1,\ldots,r.
\end{equation*}
Then, by the commutator relations
\[ [E_a,E_{-b}] =\delta_{ab}H_a, \qquad [H_a,H_{\pm b}] =\pm K_{ab}E_{\pm b}, \]
\eqref{cs2} is reduced to the following system of equations:%
\begin{equation}
  \label{cs3}
 \Delta u_{a} +\frac{v^4}{\kappa^{2}} \sum_{b=1}^{r} K_{ba}e^{u_{b}}
 - \frac{v^4}{\kappa^{2}} \sum_{b=1}^{r} \sum_{c=1}^{r} e^{u_{b}}K_{cb}e^{u_{c}}K_{ba}
 =4\pi \sum_{j=1}^{N_{a}}\delta _{p_{j}^{a}} \quad\mbox{in }~ \mathbb{R}^2,
\end{equation}%
where $a=1,\ldots,r$, ~$K=\left( K_{ab}\right) $ is the Cartan matrix of a semi-simple Lie algebra,
$\big\{ p_{j}^{a}\big\}$ are (not necessarily distinct) zeros of $\phi^{a}$, which are called vortex points.
We refer to \cite{D1,T2,yang1} for the detailed derivation from (\ref{cs2}) to (\ref{cs3}).

In this paper, we set $v^4=\kappa^2$ without loss of generality and consider only the case $r=2$,
which is the simplest among non-Abelian models.  Practically, if $r=2$, then there are only three different gauge groups, that is, $K=SU(3)$, $SO(5)$, and $G_2$.
There may exist three types of solutions to \eqref{cs3} according to their asymptotic behaviors at $\infty$ as follows:
\begin{enumerate}
 \item[(i)] $(u_1,u_2)$ is called a topological solution if  \\
 $\displaystyle\lim_{|x|\to\infty} u_a(x) = \ln \Big((K^{-1})_{1a} +(K^{-1})_{2a} \Big)$,
 ~~($a=1,2$)
 \item[(ii)] $(u_1,u_2)$ is called a non-topological solution if  \\
 $\displaystyle\lim_{|x|\to\infty} u_1(x) =\lim_{|x|\to\infty} u_2(x) =-\infty$,
 \item[(iii)] $(u_1,u_2)$ is called a mixed type solution if \\
 either $\displaystyle\lim_{|x|\to\infty} \big(u_1(x), u_2(x)\big) =(-\ln K_{11}, -\infty)$, \\
 or $\displaystyle\lim_{|x|\to\infty} \big(u_1(x),u_2(x)\big) =(-\infty, -\ln K_{22})$.
\end{enumerate}
We note that the first case (i) is valid only if $(K^{-1})_{1a} +(K^{-1})_{2a}>0$ ($a=1,2$).

The simplest case of \eqref{cs3} may be when the gauge group is Abelian, i.e. $U(1)$.
In this case, \eqref{cs3} is reduced to the following single equation.
\begin{align}
 \Delta u +e^{u}(1-e^{u}) = 4\pi \sum_{j=1}^{N_{1}}\delta_{p_{j}}
 \quad\mbox{in }~ \mathbb{R}^2.  \label{eq:CSH-top}
\end{align}
The equation \eqref{eq:CSH-top} is called the  $U(1)$  Chern-Simons Higgs equation\cite{HKP,JW} and
has been proposed in an attempt to explain high temperature superconductivity or anyonic excitations.
 \eqref{eq:CSH-top} admits only topological and nontopological solutions and has been studied extensively(see \cite{CFL,CKL,HKP,JW,LY1,NT,S-W1,S-W2} and references therein).
In particular, the existence of a topological solution of \eqref{eq:CSH-top} has been completely settled\cite{S-W1,wang} and
that of a nontopological solution has been settled almost\cite{CKL}.
Further, if all the vortex points coincide, $p_j={\bf 0}$ for all $j$,
it is known that every topological solution of \eqref{eq:CSH-top} is radially symmetric\cite{H}, unique\cite{CHMY}, and non-degenerate\cite{Ch}.

When the Cartan matrix is $K=\begin{pmatrix} 2 & -b \\ -a & 2 \end{pmatrix}$,
then the system \eqref{cs3} becomes the following nonlinear elliptic system:
\begin{equation}
  \label{eq:main}
 \left\{ \begin{aligned}
& \Delta u_1 +2e^{u_1} -ae^{u_2} -4e^{2u_1} +2ae^{2u_2} -a(b-2) e^{u_1+u_2}
 = 4\pi \sum_{j=1}^{N_{1}}\delta _{p_{j}} \\
& \Delta u_2 +2e^{u_2} -be^{u_1} -4e^{2u_2} +2be^{2u_1} -b(a-2) e^{u_1+u_2}
 = 4\pi \sum_{k=1}^{N_{2}}\delta _{q_{k}}%
\end{aligned} \right. \quad\mbox{in }~ \mathbb{R}^{2},
\end{equation}
where the constants $a$ and $b$ are given by
\[ \begin{pmatrix} a\\ b \end{pmatrix} \left(or\; \begin{pmatrix}  b \\ a \\ \end{pmatrix}  \right)
    = \begin{pmatrix}  1 \\  1 \\ \end{pmatrix},
      \quad \begin{pmatrix}  1 \\  2 \\ \end{pmatrix},
      \quad \begin{pmatrix}  1 \\  3 \\ \end{pmatrix}. \]
Each case arises from the Chern-Simons $SU(3)$, $SO(5)$, and $G_2$ models, respectively.
 In the conventional classification of
systems of equations, \eqref{eq:main} is neither cooperative nor competitive, that is, each nonlinear term in
\eqref{eq:main} is not monotone with respect to any of $u_1$ and $u_2$. This causes the main difficulty to  study
 \eqref{eq:main}. For example, unlike $U(1)$ Chern-Simons Higgs equation
 \eqref{cs2},
$L^1(\RN)$- norm   boundedness of  nonlinear terms in \eqref{eq:main} is not easy to prove
 even for  radially symmetric solutions \cite{HL}, and it  is still unknown  for non-radial solutions.

For any configuration $\{p_j^a\}$ in $\mathbb{R}^2$, Yang \cite{Yang0} proved the existence of
topological solutions of \eqref{eq:main} by the variational method and Moser-Trudinger inequality.  However, it is harder to find not only non-topological but also mixed type solutions due to logarithmic growth at infinity. Recently, there are some developments for non-topological solutions (see  \cite{ALW, CKL2, CKL3, KLL}). Meanwhile, analysis on mixed type solutions is still poor,
and only the existence results for radially symmetric mixed type solutions of \eqref{eq:main} have been
 established in \cite{CKL4,CKL5}.   In fact, mixed type solutions are not allowed in the $U(1)$ Chern-Simons theory
nor in Toda system.   Hence, it is characteristic to non-Abelian gauge theories and suggests new dynamics in these theories.
In shooting argument, radial mixed type solutions correspond to the boundary of the set of nontopological solutions\cite{Z}.  Therefore, analysis on mixed type solutions is meaningful not only due to physical reason but also to understand the non-topological solutions.  In this reason,  we shall establish the existence of mixed type solutions for any distribution of vortex points in this paper.

When the vortex points coincide, in \cite{CKL5}, they give a condition of possible bubbling for mixed type solutions.
They proved that, for each $\beta>\frac{bN_1}{2}+N_2+2$, \eqref{eq:main} admits a radially symmetric
solution $(u_1,u_2)$ such that
\[ u_1(x)\to -\ln 2 \quad\mbox{and}\quad u_2(x) =-2\beta\ln |x|+O(1) \quad\mbox{as }~ |x|\to\infty. \]
Furthermore, every radially symmetric solution $(u_1,u_2)$ of \eqref{eq:main} can be expressed as
\[ u_j(r) =2N_j\ln r +s_j +o(r) \quad\mbox{as }~ r\to 0 ~\quad (j=1,2) \]
for some $s_j\in\mathbb{R}$. It is also proved in \cite{CKL5} that if $s_2\to-\infty$ then
$\displaystyle \sup_{r\ge0} u_2(r) \to -\infty$, and $u_2(R)+2\ln R=O(1)$,
where $u_2(R)=\sup_{r\ge 0} u_2(r)$. Moreover, $\beta\to \frac{bN_1}{2}+N_2+2$ and
$u_1+\ln 2 -u \to 0$ in $C_{loc}^0([0,\infty))$, where $u$ is the radially symmetric topological solution of
the Chern-Simons equation, \eqref{eq:CSH-top}, that is, $u$ satisfies the following boundary condition:  \begin{align}
  u(x)\to 0 \quad\mbox{as }~ & |x|\to\infty.  \label{eq:CSH-topBC}
\end{align}

  However, we need  a different approach to find mixed type solutions of \eqref{eq:main} with an  arbitrary configuration of vortex points. For this purpose, the degree theory in \cite{CKL5} would be a powerful tool.  For example, to $U(1)$ Chern-Simons Higgs equation
 \eqref{cs2},  Choe, Kim, Lin in \cite{CKL} applied the degree theory and  almost completed finding solutions for an arbitrary configuration of vortex points.   To apply the degree theory to \eqref{eq:main}, as the first step, we should find when a priori bound would be broken, that is when a partially blowing up mixed type solution exists.

To find a partially blowing up mixed type solution $(u_1,u_2)$ of \eqref{eq:main},  we consider an equivalent problem by using the different scales for $(u_1,u_2)$.
As in \cite{KLL}, we introduce a small scaling parameter $\e>0$ and let
 \[ \hat{u}_j(x) =u_j(x/\e) \quad\mbox{for }~ x\in\mathbb{R}^2. \quad (j=1,2) \]
Note that \eqref{eq:main} is equivalent to the following system
\begin{align}
& \Delta u_1 =4e^{2u_1} -2e^{u_1} +ae^{u_2} -2ae^{2u_2} +a(b-2)e^{u_1+u_2} +\sum_{j=1}^{N_1}4\pi\delta_{p_j},
 \label{eq:main-u1} \\
& \Delta\Big( \frac{b}{2} \hat{u}_1 +\hat{u}_2\Big) = \frac{1}{2\e^2}(ab-4)
 (e^{\hat{u}_2} -2e^{2\hat{u}_2} +be^{\hat{u}_1 +\hat{u}_2})
 +\sum_{j=1}^{N_1} 2\pi b\delta_{\e p_j} +\sum_{k=1}^{N_2} 4\pi\delta_{\e q_k}. \label{eq:main-u2}
\end{align}
Inspired by \cite{CKL5}, we look for a family of solutions $(u_1,u_2)$ such that
\[ u_1+\ln 2 -U \to 0 \quad\mbox{in}~ C^0(\mathbb{R}^2),\] where $U$ is a topological solution of \eqref{eq:CSH-top}, and \[\hat{u}_2 -2\ln\varepsilon \to W \quad\mbox{in}~ C_{loc}^0(\mathbb{R}^2\backslash\{{\bf 0}\}) \]
for some function $W$ as $\varepsilon\to 0$. Then
 $\hat{u}_1 \to -\ln 2$ in $C_{loc}^0(\mathbb{R}^2\backslash\{{\bf 0}\})$. Hence
\[ \frac{1}{2\e^2}(ab-4) (e^{\hat{u}_2} -2e^{2\hat{u}_2} +be^{\hat{u}_1 +\hat{u}_2})
 \to \frac14(ab-4)(2+b)e^W \quad\mbox{in }~ C_{loc}^0(\mathbb{R}^2\backslash\{{\bf 0}\}). \]
So it is reasonable to choose $W$ as a solution of   the Liouville equation:
\begin{equation}
  \label{eq:W-equation}
 \left \{  \begin{aligned}
& \Delta W +\frac14(4-ab)(2+b) e^W =2\pi (bN_1+2N_2) \delta_{\bf 0}, \\
&~ e^{W} \in L^1(\mathbb{R}^2).
\end{aligned} \right.
\end{equation}
The arguments above give us some motivation to  construct  a partially blowing up mixed type solution. Indeed, we have the following result.
\begin{theorem}
  \label{T121}
Assume that \eqref{eq:CSH-top} admits a non-degenerate topological solution $U(x)$.
Suppose one of the following conditions holds.
\begin{enumerate}
 \item $bN_1 +2N_2 \ge 3$, or
 \item $bN_1 +2N_2 \le 2$ and $p_j=q_k=\{{\bf 0}\}$ for all $j$ and $k$.
\end{enumerate}
Then, there exists a constant $\e_0>0$ such that for each $\e\in (0,\e_0)$, the system \eqref{eq:main}
has a mixed type solution $( u_{1,\e}, u_{2,\e})$ such that
\[ u_{1,\e}(x) \to -\ln 2 \quad\mbox{and}\quad u_{2,\e}(x)=-2\beta_{\e}\ln|x|+O(1)
 \quad\mbox{as }~ |x| \to\infty. \]
for some $\beta_\e\in\mathbb{R}$, where $\displaystyle \beta_\e =\frac{bN_1}{2}+N_2+2 +O(\e^2)$ as $\e\to 0$.

Moreover, as $\e\to 0$, $(u_{1,\e},u_{2,\e})$ satisfies
\begin{align*}
& u_{1,\e} -U +\ln 2 \to 0 ~\mbox{ in } C^0(\mathbb{R}^2) \quad\mbox{and} \\
& \Big(u_{2,\e} +\frac{b}{2}U \Big)(\cdot/\e) -2\ln\e \to W \mbox{ in } C^0_{loc}(\RN)
\end{align*}
 where $W$  is a solution of \eqref{eq:W-equation}.
\end{theorem}
For some technical reason, we assume that \eqref{eq:CSH-top} admits
a non-degenerate solution.  Here, by nondegeneracy of a solution $U$, we mean that {\it the linearized operator
$$\displaystyle\mathcal{L}_1=\Delta +e^U(1-2e^U)$$
is a continuous bijection from $H^2(\RN)$ onto $L^2(\RN)$, and the inverse operator $\mathcal{L}_1^{-1}$ is also continuous}.
However, this nondegeneracy condition is reasonable counting on the general transversality theorem (See for example theorem 1.7.5 in \cite{Niren}).  In fact, if either $\max_{1\le j\le N_1} |p_j|$ is sufficiently small or
$\min_{1\le j<k\le N_1}|p_j-p_k|$ is sufficiently large
then \eqref{eq:CSH-top} admits a unique topological solution, which is non-degenerate \cite{Ch}.  Therefore Theorem \ref{T121} extends the results in \cite{CKL4,CKL5} to an arbitrary configuration of $\{p_j\}$
as long as $U$ is non-degenerate and the decay rate is small enough.


It is interesting to see that $u_{1,\e}$ converges in itself while $u_2$ converges after a suitable scaling.  This
means they live in different scalings.  Due to the boundary condition at infinity, one might want to choose an approximate solution $\tilde{V}_{1,\e} =-\ln 2 +U$ for $u_{1,\e}$.
But it turns out that $\tilde{V}_{1,\e} =-\ln 2 +U$ is not accurate enough since  $u_{2,\e}$ shows bubbling phenomena near $\infty$.  Indeed,  $\tilde{V}_{1,\e} =-\ln 2 +U$ cannot balance the mass contribution  of $e^{u_{2,\e}}$ from  $\infty$,
since $U$ decays exponentially fast near $\infty$.


To overcome this difficulty,
we should compare an effect from  $2e^{u_{1,\varepsilon}}(1-2e^{u_{1,\varepsilon}})$ and
an effect from $e^{u_{2,\varepsilon}}$ to construct a suitable approximate solution for $u_{1,\e}$.
We remark that the similar situation also occurs in \cite{KLL},
where they overcome the difficulty by refining the errors with the additional term
$\frac12(W^*({\bf 0}) -W^*(\e x))$, where $W^*$ is the regular part of the solution $W$ of \eqref{eq:W-equation}.
However, in our case the term $\frac12(W^*({\bf 0}) -W^*(\e x))$ is not appropriate,
since it grows logarithmically near $\infty$.  To remove this obstacle,
we note that if $f(t):=e^t(1-e^t)$, then $f(u_{1,\e}+\ln 2)=f(0)+f'(0)(u_{1,\e}+\ln 2)+O(|u_{1,\e}+\ln 2|^2)$, $f(0)=0$, and  $f'(0)=-1$. It implies that $u_{1,\e}(x)+\ln 2$ should be close to $-\frac{ab}{2}e^{u_{2,\e}\left(
 x\right)}$ when $|x|\gg1$  to balance the mass contribution of $e^{u_{2,\varepsilon}}$ at infinity.

In conclusion, we are going to use a combination of topological solution $U$ of \eqref{eq:CSH-top} and  $-\frac{ab}{2}\e^2e^{W\left(\e
 x\right)}$  together as an approximate solution for $u_{1,\varepsilon}$ (see the exact form of  the approximate solution in \eqref{eq:V-approximation}) and derive the correct finite dimensional reduced problem.
Then, we shall show the finite dimensional reduced problem is invertible in a suitable space and find a family of mixed type solutions.

This paper is organized as follows. In Section 2, we introduce an approximate solution and
review useful properties of the linearized operator. In Section 3, we present the proof of Theorem \ref{T121}.


\setcounter{equation}{0}
\section{Basic Estimates: Approximation Solutions}

For simplicity, we let
\begin{equation}
  \label{eq:f}
 f(t) =e^t(1-e^t) \quad\mbox{for }~ t\in\mathbb{R}.
\end{equation}
We now recall some well-known results. If $U$ is a solution of \eqref{eq:CSH-top} then $U\le 0$ in $\mathbb{R}^2$.
Moreover, there exist constants $C_0,R_0>1$, which may depend on $U$, such that
\[ |U(x)|+|\nabla U(x)| \le C_0e^{-|x|} \quad\mbox{for }~ |x|\ge R_0. \]

Every solution of the Liouville equation \eqref{eq:W-equation}
is completely classified by Prajapat and Tarantello \cite{PT}, and it takes the form
\begin{equation}
  \label{eq:W}
 W_{\mu,\alpha}(z)= \ln\frac{32e^\mu \lambda^2|z|^{2\lambda-2}}{(4-ab)(2+b)
 (1+e^{\mu}|z^\lambda +\alpha|^2)^2}, ~\quad z=x_1+{\rm i}x_2 \in\mathbb{C},
\end{equation}
where $\alpha\in\mathbb{C}$ and $\mu\in\mathbb{R}$ are parameters, and
\[ \lambda= \frac{bN_1}{2} +N_2 +1. \]
Recall that $\alpha=0$ if $\lambda\notin\mathbb{N}$.
To simplify notations, we write
\[ W_\alpha(z) =W_{0,\alpha}(z) \]
and
\[ {W}_\alpha^*(z) =W_\alpha(z) -(2\lambda-2)\ln |z|. \]
\subsection{Function spaces}
We introduce some function spaces we will work on.
Let
\[ \sigma(x) =1+|x|, ~\quad x\in\mathbb{R}^2. \]
Fix a constant $0<d<1/4$. We define the function space $X$ by
\[ X=\{ v\in H_{loc}^2(\mathbb{R}^2) \mid \|v\|_X<\infty\}, \]
where
\[ \|v\|_X^2 =\big\| \sigma^{1+d}\Delta v\big\|_{L^2(\mathbb{R}^2)}^2
 +\big\| \sigma^{-1-d}v\big\|_{L^2(\mathbb{R}^2)}^2. \]
We define the function space $Y$ by
\[ Y=\{h\in L_{loc}^2(\mathbb{R}^2) \mid \|h\|_Y :=\big\| \sigma^{1+d}h\big\|_{L^2(\mathbb{R}^2)} <\infty\}, \]
We also define two inner products $(\cdot,\cdot)_{\LN}$ and $(\cdot,\cdot)_Y$ as follows.
\begin{align*}
 (u,\xi)_{\LN} &= \int_{\RN} u\xi dx, ~\quad u,\xi\in \LN, \\
 (\eta,h)_Y &= \int_{\RN} \sigma^{2+2d}\eta h dx, ~\quad \eta, h\in Y.
\end{align*}

For $z=x_1+{\rm i}x_2 \in\mathbb{C}$, we define
\begin{align*}
 Z_{\alpha,0}(z) &= \frac{\partial}{\partial\mu}\Big|_{\mu=0} W_{\mu,\alpha}(z)
 =\frac{1-|z^\lambda +\alpha|^2}{1+|z^\lambda +\alpha|^2}, \\
 Z_{\alpha,1}(z) &= -\frac12 \mbox{Re}\Big( \frac{\partial W_\alpha}{\partial\overline{\alpha}}\Big)(z)
 \Big|_{\mu=0} =\frac{\mbox{Re}(z^\lambda +\alpha)}{1+|z^\lambda +\alpha|^2}, \\
 Z_{\alpha,2}(z) &= -\frac12 \mbox{Im}\Big( \frac{\partial W_\alpha}{\partial\overline{\alpha}}\Big)(z)
 \Big|_{\mu=0} =\frac{\mbox{Im}(z^\lambda +\alpha)}{1+|z^\lambda +\alpha|^2},
\end{align*}
where $\mbox{Re}$ and $\mbox{Im}$ denote the real and imaginary parts, respectively.
It is easily checked that $Z_{\alpha,\mu}\in X$ for $\mu=0,1,2$.
Moreover, $\sigma^{-2-2d}Z_{\alpha,j}\in Y$ and $(h,Z_{\alpha,j})_{\LN}=(h,\sigma^{-2-2d}Z_{\alpha,j})_Y$
for $j=1,2$.

We now introduce a subspace $E_\alpha$ of $X$ as follows. We define
\begin{equation*}
\begin{aligned}E_\alpha =\left\{ \begin{array}{ll}  &\{\xi\in X\mid \big(\xi, \, e^{W_\alpha} Z_{\alpha,i}\big)_{\LN}=0 \quad(i=0,1,2)\}
 \quad\mbox{if }~ \lambda \in\mathbb{N},
\\&\{\xi\in X\mid \big(\xi, \, e^{W_0} Z_{0,0}\big)_{\LN} =0\}
 \quad\mbox{if }~ \lambda\notin \mathbb{N}. \end{array}\right. \end{aligned}
\end{equation*}

We also introduce a subspace $F_\alpha$ of $Y$ as follows. We define
\begin{equation*}
\begin{aligned}F_\alpha =\left\{ \begin{array}{ll}  &\{ h\in Y\mid \big(h, Z_{\alpha,i}\big)_{\LN} =0 \quad (i=1,2)\}
 \quad\mbox{if }~ \lambda\in\mathbb{N},
\\&Y \quad\mbox{if }~ \lambda\notin\mathbb{N}. \end{array}\right. \end{aligned}
\end{equation*}
\begin{lemma}
  \label{lem:T-aux}
Suppose $\lambda\in\mathbb{N}$. There exists a constant $\e_0>0$ such that
if $|\alpha|<\e_0$ then for each $h\in Y$ then there exists a unique pair of constants
$(c_{\alpha,1}, c_{\alpha,2}) \in\mathbb{R}^2$ satisfying
\[ h -c_{\alpha,1}\sigma^{-2-2d}Z_{\alpha,1} -c_{\alpha,2}\sigma^{-2-2d}Z_{\alpha,2} \in F_\alpha. \]
\end{lemma}
\begin{proof}
Note that $h -c_{\alpha,1}\sigma^{-2-2d}Z_{\alpha,1} -c_{\alpha,2}\sigma^{-2-2d}Z_{\alpha,2} \in F_\alpha$
if and only if
\[ \int_{\mathbb{R}^2} hZ_{\alpha,j}dx =c_{\alpha,1}\int_{\RN} \sigma^{-2-2d}Z_{\alpha,1}Z_{\alpha,j}dx
 +c_{\alpha,2}\int_{\mathbb{R}^2} \sigma^{-2-2d}Z_{\alpha,2}Z_{\alpha,j}dx, ~\quad (j=1,2) \]
or equivalently,
\begin{equation}
  \label{eq:T-c}
 \begin{pmatrix} a_{11}(\alpha) & a_{12}(\alpha) \\ a_{21}(\alpha) & a_{22}(\alpha) \end{pmatrix}
 \begin{pmatrix} c_{\alpha,1} \\ c_{\alpha,2} \end{pmatrix} =
 \begin{pmatrix} b_1(h,\alpha) \\ b_2(h,\alpha) \end{pmatrix},
\end{equation}
where we set $\displaystyle a_{jk}(\alpha)=\int_{\mathbb{R}^2} \sigma^{-2-2d}Z_{\alpha,j}Z_{\alpha,k}dx$
and $\displaystyle b_j(h,\alpha) =\int_{\mathbb{R}^2} hZ_{\alpha,j}dx$ for simplicity.
It is easily checked that $a_{12}(0)=\alpha_{21}(0)=0$ and
\[ a_{11}(0)=a_{22}(0)= \int_0^\infty \frac{\pi r^{2\lambda+1}}{(1+r^{2\lambda})^2\sigma^{2+2d} } dr >0. \]
Consequently $a_{11}(\alpha)a_{22}(\alpha) -a_{12}(\alpha)a_{21}(\alpha)>0$ if $|\alpha|$ is sufficiently small,
which proves Lemma \ref{lem:T-aux}.
\end{proof}

For $|\alpha|<\e_0$, we define a projection map $T_\alpha:Y\to F_\alpha$ by
\begin{equation}
  \label{eq:T-definition}
 T_\alpha h = \left\{ \begin{array}{ll}
 h -c_{\alpha,1}\sigma^{-2-2d}Z_{\alpha,1} -c_{\alpha,2}\sigma^{-2-2d}Z_{\alpha,2}, &~ \lambda\in\mathbb{N}, \\
 h, &~ \lambda\notin\mathbb{N},
\end{array} \right.
\end{equation}
where the constants $c_{\alpha,1}$ and $c_{\alpha,2}$ are chosen so that \eqref{eq:T-c} holds.
Lemma \ref{lem:T-aux} implies that $T_\alpha$ is well defined if $|\alpha|<\e_0$.

\begin{lemma}
  \label{lem:T-bounded}
If $|\alpha|<\e_0$, there exists a constant $c=c(p_j,q_k)>0$ such that
\[ \|T_\alpha h\|_Y \le c\|h\|_Y \quad\mbox{for all }~ h\in Y. \]
\end{lemma}
\begin{proof}
The case $\lambda\notin\mathbb{N}$ is trivial. Thus we assume that $\lambda\in\mathbb{N}$.
It follows from \eqref{eq:T-c} that
\begin{align*}
 |c_{\alpha,j}| &\le C\big( |b_1(h,\alpha)|+|b_2(h,\alpha)|\big) \\
&\le C\|h\|_Y \big(\|\sigma^{-1-d}Z_{\alpha,1}\|_{L^2(\mathbb{R}^2)}
 +\|\sigma^{-1-d}Z_{\alpha,2}\|_{L^2(\mathbb{R}^2)}\big) \le C\|h\|_Y.
\end{align*}
Therefore we obtain that
\[ \|T_\alpha h\|_Y \le \|h\|_Y +|c_{\alpha,1}| \big\|\sigma^{-2-2d}Z_{\alpha,1}\big\|_Y
 +|c_{\alpha,2}|\big\|\sigma^{-2-2d}Z_{\alpha,2}\big\|_Y \le c\|h\|_Y, \]
which finishes the proof.
\end{proof}

\subsection{Linearized operators}

We define the operator $\mathcal{L}_1:H^2(\mathbb{R}^2)\to L^2(\mathbb{R}^2)$ by
\[ \mathcal{L}_1u =\Delta u +f'(U)u, \]
where $f$ is defined in \eqref{eq:f}, and $U$ is a solution of \eqref{eq:CSH-top}.

We also define the operator $\mathcal{L}_{2,\alpha}:X\to Y$ by
\[ \mathcal{L}_{2,\alpha}v = \Delta v +\frac14 (4-ab)(2+b)e^{W_\alpha} v. \]
Recall that $\alpha=0$ if $\lambda\notin\mathbb{N}$.

In the following lemma, we recall the kernel of $\mathcal{L}_{2,\alpha}$.

\begin{lemma}
  \label{lem:L2-kernel}
If $\lambda\in\mathbb{N}$ then $\mbox{ker}\mathcal{L}_{2,\alpha} =\mbox{span}
\{ Z_{\alpha,0}, Z_{\alpha,1}, Z_{\alpha,2} \}$.
If $\lambda \notin\mathbb{N}$ then $\mbox{ker}\mathcal{L}_{2,0} =\mbox{span}\{ Z_{0,0}\}$.
\end{lemma}
\begin{proof}
See \cite{DEM} and \cite{CL}(Lemma 2.1) for the cases $\lambda\in\mathbb{N}$ and
$\lambda\notin\mathbb{N}$, respectively. Actually, if $u\in X$ then $u(x)= c_u\ln (1+|x|) +O(1)$ as
$|x|\to\infty$ for some constant $c_u\in\mathbb{R}$ (\cite{M}).
Hence the arguments in \cite{DEM,CL} are still valid here.
\end{proof}

For $|\alpha|\le \e_0$, we define the map
$\mathbb{L}_\alpha:H^2(\mathbb{R}^2)\times E_\alpha \to L^2(\mathbb{R}^2)\times F_\alpha$ by
\begin{equation}
  \label{eq:operator-L}
 \mathbb{L}_\alpha(u,v) =\big(\mathcal{L}_1u, ~\mathcal{L}_{2,\alpha}v \big).
\end{equation}
We recall the following result.
\begin{theorem}\cite{Ch,KLLY2}
  \label{thm:L-isomorphism}
Assume that $U$ is a non-degenerate topological solution of \eqref{eq:CSH-top}.
There exists a constant $\overline{\e}_1>0$ such that if $|\alpha|<\overline{\e}_1$ then
$\mathbb{L}_\alpha$ is an isomorphism from $H^2(\mathbb{R}^2)\times E_\alpha$ onto
$L^2(\mathbb{R}^2)\times F_\alpha$. Moreover, there exists a constant $C=C(p_j,q_k)>0$ such that
\begin{align*}&\|u\|_{H^2(\mathbb{R}^2)}\le C\|\mathcal{L}_1u\|_{L^2(\mathbb{R}^2)}\quad\mbox{for all }~ u\in H^2(\mathbb{R}^2),
\\&\|v\|_X \le C\|\mathcal{L}_{2,\alpha}v\|_Y \quad\mbox{for all }~ v\in E_\alpha. \end{align*}
\end{theorem}


\setcounter{equation}{0}
\section{Existence of Solutions} In this section, we are going to prove Theorem \ref{T121}.
For a technical reason, we divide the proof of Theorem \ref{T121}  into two cases
$\displaystyle \lambda =\frac{bN_1}{2} +N_2 +1\ge \frac32$ and $\lambda=1$ since $N_1, N_2\in\mathbb{N}\cup\{0\}$.

\subsection{The case $\lambda\ge 3/2$}
We introduce some functions to simplify notations. Let
\begin{equation}
  \label{eq:PeQe}
 P_\e(x) =\prod_{j=1}^{N_1} |x-\e p_j| \quad\mbox{and}\quad Q_\e(x)=\prod_{k=1}^{N_2} |x-\e q_k|
 \quad\mbox{for }~ x=(x_1,x_2)\in\mathbb{R}^2.
\end{equation}
We let $P_\e\equiv 1$ if $N_1=0$. We also let
\begin{equation}
  \label{eq:varphi-e}
 \varphi_\e(z) =-\frac12 ab e^{W_\alpha(\e z)}\chi(z),
 ~\quad z=x_1+{\rm i}x_2\in\mathbb{C},
\end{equation}
where $\chi$ is a smooth cut-off function such that $0\le \chi\le 1$ in $\mathbb{R}^2$, and
\[ \chi(x) =\left\{ \begin{array}{ll} 0, &~ |x|\le 1/2, \\ 1, &~ |x|\ge 1. \end{array} \right. \]

We now introduce an approximate solution to \eqref{eq:main}.
For $\e>0$ and $\alpha\in\mathbb{C}$, we define a pair of functions
$(V_{1,\e}, V_{2,\e,\alpha})$ by
\begin{equation}
  \label{eq:V-approximation}
 \left\{ \begin{aligned}
 V_{1,\e}(x) &= U(x) -\ln 2 +\e^2\varphi_\e(x), \\
 V_{2,\e,\alpha}(x) &= W_\alpha^*(\e x) +b\ln P_\e(\e x) +2\ln Q_\e(\e x)
 -\frac{b}{2} U(x) -\frac{b}{2}\e^2\varphi_\e(x),
 \end{aligned} \right.
\end{equation}
where $U$ is a non-degenerate topological solution of \eqref{eq:CSH-top}.
We use $(V_{1,\e}, V_{2,\e,\alpha}+2\ln\e)$ as an approximate solution to \eqref{eq:main}.
As we mentioned before, $\varphi_\varepsilon$ is added to $V_{1,\varepsilon}$ to cover
the mass contribution of $ae^{u_2} +a(b-2)e^{u_1+u_2}$ in the first equation.

If $\e>0$ is sufficiently small, we will find a solution $(u_1,u_2)$ of \eqref{eq:main} of the form
\begin{equation}
  \label{eq:u-form}
 \left\{ \begin{aligned}
  u_1(x) &= V_{1,\e}(x) +\e^2 \xi_{\e,\alpha}(x), \\
  u_2(x) &= V_{2,\e,\alpha}(x) +2\ln\e -\frac{b}{2}\e^2\xi_{\e,\alpha}(x) +\e^2\eta_{\e,\alpha}(\e x),
 \end{aligned} \right.
\end{equation}
for some $\alpha=\alpha(\e) \in\mathbb{C}$. Here $\e^2 \xi_{\e,\alpha}(x)$ and $\e^2\eta_{\e,\alpha}(\e x)$ are error terms. It will turn out that $|\alpha(\e)|=o(1)$,
$\|\xi_{\e,\alpha(\e)}\|_{H^2(\RN)}=o(1)$ and $\|\eta_{\e,\alpha(\e)}\|_X=O(1)$ as $\e\to 0$. We note that there is a constant $c_0>0$ satisfying
\begin{equation}
  \label{eq:xi-eta-Linfty0}
 \begin{aligned} |\xi(x)| &\le c_0\|\xi\|_{H^2(\RN)}  \quad\mbox{for all}
  \ \ \xi\in H^2(\RN), \ \mbox{and} \\
 |\eta(x)| &\le c_0\|\eta\|_{X}\big( 1+\ln\sigma(x)\big) \quad\mbox{for all }~ x\in\RN, \ \eta\in X,  \end{aligned}
\end{equation}here we used $W^{2,2}$ estimation and  \cite[Theorem 4.1]{CFL} respectively.
Together with $V_{2,\e,\alpha}(x) =-(bN_1 +2N_2 +4)\ln |x|+O(1)$ as $|x|\to\infty$,
we will obtain the limit of $\beta_\e$ as in Theorem \ref{T121}.

We rewrite the system \eqref{eq:main-u1}-\eqref{eq:main-u2} as
\begin{align}
 \mathcal{L}_1{\xi}_{\e,\alpha} &= g_{1,\e,\alpha}(\xi_{\e,\alpha},\eta_{\e,\alpha}),
  \label{eq:eq-L1} \\
 \mathcal{L}_{2,\alpha} \eta_{\e,\alpha} &= g_{2,\e,\alpha}(\xi_{\e,\alpha}, \eta_{\e,\alpha}),
  \label{eq:eq-L2}
\end{align}
where $g_{1,\e,\alpha}$ and $g_{2,\e,\alpha}$ are defined by
\begin{equation}
  \label{eq:g-1e}
\begin{aligned}
 g_{1,\e,\alpha}(\xi,\eta)(x)
&= -\frac{1}{\e^2}\Big( f(U+\e^2\varphi_\e +\e^2\xi) -f(U) -\e^2f'(U)(\varphi_\e +\xi)\Big)(x) \\
&\quad -\Delta\varphi_\e(x) -f'(U(x))\varphi_\e(x) \\
&\quad +a \exp\Big(\Big(V_{2,\e,\alpha} -\frac{b}{2}\e^2\xi\Big)(x) +\e^2\eta(\e x)\Big) \\
&\quad +a(b-2)\exp\Big( \Big(V_{1,\e}+V_{2,\e,\alpha} +\frac{2-b}{2}\e^2\xi\Big)(x)
 +\e^2\eta(\e x)\Big) \\
&\quad -2a\e^2 \exp\big(\big(2V_{2,\e,\alpha} -b\e^2\xi\big)(x) +2\e^2\eta(\e x)\big)
\end{aligned}
\end{equation}
and
\begin{equation}
  \label{eq:g-2e}
\begin{aligned}
 g_{2,\e,\alpha}(\xi,\eta)(x)
&= -\frac{4-ab}{2\e^2} \exp\Big( \Big(V_{2,\e,\alpha} -\frac{b}{2}\e^2\xi\Big) (x/\e) +\e^2\eta(x)\Big) \\
&\quad -\frac{b(4-ab)}{2\e^2} \exp\Big( \Big(V_{1,\e} +V_{2,\e,\alpha} +\frac{2-b}{2}\e^2\xi\Big)
 (x/\e) +\e^2\eta(x)\Big) \\
&\quad +\frac{1}{4\e^2}(4-ab)(2+b)e^{W_\alpha(x)} \big(1+\e^2\eta(x) \big) \\
&\quad +(4-ab)\exp\big( \big(2V_{2,\e,\alpha} -b\e^2\xi\big)(x/\e) +2\e^2\eta(x) \big).
\end{aligned}
\end{equation}

By a shift of origin, without loss of generality, throughout this paper, we always assume that
\begin{equation}
  \label{eq:PQ-location}
 \sum_{j=1}^{N_1} bp_j +\sum_{k=1}^{N_2} 2q_k ={\bf 0}.
\end{equation}

We define
\[ S_0=\{(\xi,\eta)\in H^2(\RN)\times E_\alpha\mid \|\xi\|_{H^2(\RN)} +\|\eta\|_X\le M_0\}, \]
where $M_0\ge 1$ is a constant to be determined later.
Recall the map $\mathbb{L}_\alpha$ defined in \eqref{eq:operator-L}.

\begin{proposition}
  \label{prop:contraction}
Let $U$ be  a non-degenerate topological solution of \eqref{eq:CSH-top}.
There exist constants $M_0\ge 1$ and $\overline{\e}_2>0$ satisfying the following property:
if $0<\e<\overline{\e}_2$ and $|\alpha|<\overline{\e}_2$ then there exists a unique element
$(\xi_{\e,\alpha},\eta_{\e,\alpha})\in S_0$ such that
\begin{equation}
  \label{eq:L-g-solution}
 \mathbb{L}_\alpha(\xi_{\e,\alpha},\eta_{\e,\alpha})
 =\big(g_{1,\e,\alpha}(\xi_{\e,\alpha},\eta_{\e,\alpha}),~
 T_\alpha g_{2,\e,\alpha}(\xi_{\e,\alpha},\eta_{\e,\alpha})\big).
\end{equation}
\end{proposition}
\begin{proof}
The proof is based on the contraction mapping theorem.
By \eqref{eq:xi-eta-Linfty0}, we have  if $(\xi,\eta)\in S_0$ then
\begin{equation}
  \label{eq:xi-eta-Linfty}
 \begin{aligned} |\xi(x)| &\le   c_0M_0 \ \ \mbox{and}\ \
 |\eta(x)| \le c_0M_0\big( 1+\ln\sigma(x)\big) \quad\mbox{for all }~ x\in\RN.  \end{aligned} \end{equation}
In this proof, we will denote by $C$ and $C_i$ various constants
independent of $\e$, $\alpha$ and $(\xi,\eta)\in S_0$. We let
\begin{equation}
  \label{eq:R0}
 R_0= 1 +5\max_{j,k}\{|p_j|, |q_k|\}.
\end{equation}
For $|x|\ge R_0\e$, we define $H_\e(x)$ by
\begin{align}
 H_\e(x) &= b\ln P_\e(x) +2\ln Q_\e(x) -(bN_1+2N_2)\ln |x|  \label{eq:He} \\
&= \sum_{j=1}^{N_1}\frac{b}{2} \ln \Big(1-\frac{2\e p_j\cdot x}{|x|^2} +\frac{\e^2|p_j|^2}{|x|^2} \Big)
 +\sum_{k=1}^{N_2} \ln \Big(1-\frac{2\e q_k\cdot x}{|x|^2} +\frac{\e^2|q_k|^2}{|x|^2} \Big). \nonumber
\end{align}
Then $|H_\e(x)|\le C\e^2/|x|^2$ for $|x|\ge R_0\e$. See \eqref{eq:A}-\eqref{eq:He-A} below.

We claim that there exist constants ${\e}'={\e}'(M_0)>0$ and $C_1=C_1(p_j,q_k)>0$
such that if $0<\e<{\e}'$ and $|\alpha|\le 1$ then
\begin{equation}
  \label{eq:g1e-defined}
 \|g_{1,\e,\alpha}(\xi,\eta)\|_{\LN} \le C_1M_0e^{C_1M_0\e^2} \e
 \quad\mbox{for all }~ (\xi,\eta)\in S_0.
\end{equation}
To prove \eqref{eq:g1e-defined}, we write
\[ g_{1,\e,\alpha}(\xi,\eta)(x) =I_1 +I_2 +I_3 +I_4 +I_5, \]
where
\begin{align*}
 I_1 &= -\frac{1}{\e^2}\Big( f(U+\e^2\varphi_\e +\e^2\xi) -f(U) -\e^2f'(U)(\varphi_\e +\xi)\Big)(x), \\
 I_2 &= -\Delta\varphi_\e(x) -\big(f'(U(x)) +1\big)\varphi_\e(x), \\
 I_3 &= ae^{W_\alpha^*(\e x) +b\ln P_\e(\e x) +2\ln Q_\e(\e x)
 -\frac{b}{2}(U+\e^2\varphi_\e +\e^2\xi)(x) +\e^2\eta(\e x)} \big( 1-e^{(U+\e^2\varphi_\e +\e^2\xi)(x)}\big), \\
 I_4 &= \frac12 ab \big[ e^{W_\alpha^*(\e x) +b\ln P_\e(\e x) +2\ln Q_\e(\e x)
 +\frac{2-b}{2}(U+\e^2\varphi_\e +\e^2\xi)(x) +\e^2\eta(\e x)} -e^{W_\alpha(\e x)}\chi(x) \Big], \\
 I_5 &= -2a\e^2 e^{2W_\alpha^*(\e x) +2b\ln P_\e(\e x) +4\ln Q_\e(\e x)
 -b(U+\e^2\varphi_\e +\e^2\xi)(x) +2\e^2\eta(\e x)}.
\end{align*}
Note that
\[ I_1 = -\frac{1}{\e^2}e^U\big(e^{\e^2(\varphi_\e +\xi)} -1-\e^2(\varphi_\e +\xi)\big)
 +\frac{1}{\e^2} e^{2U}\big(e^{2\e^2(\varphi_\e +\xi)} -1-2\e^2(\varphi_\e +\xi)\big). \]
By \eqref{eq:xi-eta-Linfty} and the inequality $|e^t-1-t|\le (1/2)e^{|t|}|t|^2$, we obtain
\[ \|I_1\|_{\LN} \le C e^{CM_0\e^2}(\e +M_0^2\e^2). \]

Note that
\[ \big|f'(U(x))+1\big| =|e^{U(x)}-1||2e^{U(x)}+1| \le Ce^{-|x|} \quad\mbox{for }~ |x|\ge 1/2. \]
Since $e^{W_\alpha(\e x)}\le C\e^{2\lambda-2}|x|^{2\lambda-2}$ and $\lambda\ge 3/2$, it follows that
\begin{align*}
 \|I_2\|_{\LN}^2 & \le C\|\Delta\varphi_\e\|_{\LN}^2 +C\int_{|x|\ge 1/2} e^{-2|x|}e^{2W_\alpha(\e x)} dx \\
&\le C\e^2 +C\e^{4\lambda-4} \int_{|x|\ge 1/2} |x|^{4\lambda -4} e^{-2|x|} dx \\
&\le C(\e^2 +\e^{4\lambda-4}) \le C\e^2 \quad\mbox{for }~ 0<\e<1 ~\mbox{ and }~ |\alpha|\le 1.
\end{align*}

We now estimate $\|I_k\|_{\LN}$ ($k=3,4,5$). For this purpose, we assume $|\alpha|\le 1$
and we divide $\RN$ into two regions $\{x\mid |x|\le R_0\}$ and $\{x\mid |x|\ge R_0\}$.

Since $U(x)-2\sum_{i=1}^{N_1}\ln |x-p_i|\in C^{\infty}(B_{R_0}(0))$, we have \begin{equation}\label{smallupp}e^{2\ln P_\e(\e x) -U(x)}\le C\e^{2N_1}\ \ \textrm{ for}\ \  |x|\le R_0.\end{equation} Together with $\lambda\ge 3/2$,
it follows from \eqref{eq:xi-eta-Linfty} that
\[ |I_3|+|I_4|+|I_5| \le Ce^{CM_0\e^2}\e^{2\lambda-2} \le Ce^{CM_0\e^2}\e
 \quad\mbox{for }~ |x|\le R_0, \]
and hence $\|I_k\|_{L^2(|x|\le R_0)}\le Ce^{CM_0\e^2}\e$. ($k=3,4,5$)

For $|x|\ge R_0$, we have
\[ W_\alpha^*(\e x)+b\ln P_\e(\e x)+2\ln Q_\e(\e x) =W_\alpha(\e x)+H_\e(\e x). \]
Then it follows from \eqref{eq:xi-eta-Linfty} that if $|x|\ge R_0$ then
$\e^2|\eta(\e x)| \le c_0M_0\e^2(1+\ln\sigma(\e x))$.
Thus if $|x|\ge R_0$ then by the inequality $|e^t-1|\le e^{|t|}|t|$ for $t\in\mathbb{R}$,
\begin{align*}
 |I_3|+ |I_4| &\le Ce^{CM_0\e^2} (\sigma^{c_0M_0\e^2}e^{W_\alpha})(\e x) \Big(|H_\e(\e x)|+|U(x)| \\
&\qquad +\e^2|\varphi_\e(x)| +\e^2|\xi(x)| +\e^2|\eta(\e x)|\Big) \\
&\le Ce^{CM_0\e^2} (\sigma^{c_0M_0\e^2}e^{W_\alpha})(\e x) \Big(\frac{1}{|x|^2} +e^{-|x|} +M_0\e^2
 +M_0\e^2\ln\sigma(\e x) \Big).
\end{align*}
If in addition that $\lambda =3/2$ then $N_2=0$ and $b=N_1=1$. Then we have $p_1={\bf 0}$
by the assumption \eqref{eq:PQ-location}. In this case, it follows that $H_\e=0$ identically, and hence
\[ |I_3|+ |I_4| \le Ce^{CM_0\e^2} (\sigma^{c_0M_0\e^2}e^{W_\alpha})(\e x)
 \Big( e^{-|x|} +M_0\e^2 +M_0\e^2\ln\sigma(\e x) \Big). \]

Choose a constant ${\e}'>0$ such that
\[ c_0M_0(\e')^2\le 1. \]
In particular, $\e'<1$. If $0<\e<\e'$ and $|\alpha|\le 1$ then
\[ (\sigma^{c_0M_0\e^2} e^{W_\alpha})(\e x) \le C\e^{2\lambda-2} |x|^{2\lambda-2} \sigma^{1-4\lambda}(\e x)
 \quad\mbox{for }~ |x|\ge R_0, \]
and consequently $\displaystyle \|(\sigma^{c_0M_0\e^2}e^{W_\alpha})(\e x)e^{-|x|}\|_{L^2(|x|\ge R_0)}
 \le C\e^{2\lambda-2} \le C\e$.

If $\lambda\ge 2$ then
\[ \int_{|x|\ge R_0} (\sigma^{2c_0M_0\e^2}e^{2W_\alpha})(\e x) |x|^{-4} dx
 \le \int_{|y|\ge R_0\e} C\e^2 |y|^{4\lambda-8}\sigma^{2-8\lambda}(y) dy \le C\e^2. \]
Therefore $\|I_3\|_{\LN} +\|I_4\|_{\LN}\le CM_0e^{CM_0\e^2}\e$ for $0<\e<{\e}'$ and $|\alpha|\le 1$.

Clearly
\[ |I_5|\le C\e^2e^{CM_0\e^2} (\sigma^{2c_0M_0\e^2} e^{2W_\alpha})(\e x) \quad\mbox{for }~ |x|\ge R_0, \]
and hence $\|I_5\|_{\LN} \le Ce^{CM_0\e^2}\e$ for $0<\e<{\e}'$ and $|\alpha|\le 1$.
Putting all the estimates for $I_k$ together, we obtain \eqref{eq:g1e-defined}.

We claim that there exists a constant $C_2=C_2(p_j,q_k)>0$ such that if $0<\e<\e'$ and $|\alpha|\le 1$ then
\begin{equation}
  \label{eq:g2e-defined}
 \|g_{2,\e,\alpha}(\xi,\eta)\|_Y \le C_2e^{C_2M_0\e^2}(1+M_0^2\e^2) \quad\mbox{for all }~ (\xi,\eta)\in S_0.
\end{equation}
To prove \eqref{eq:g2e-defined}, we   note that \eqref{smallupp} yields
$e^{2\ln P_\e(x)-U(x/\e)}\le C\e^{2N_1}$ for $|x|\le R_0\e$. Then it follows from \eqref{eq:V-approximation}
that $\exp({V}_{2,\e,\alpha}(x/\e))\le C\e^{bN_1+2N_2}$ for $|x|\le R_0\e$, and consequently
\begin{equation}
  \label{g2bdd}\big| g_{2,\e,\alpha}(\xi,\eta)(x) \big| \le Ce^{CM_0\e^2}\e^{bN_1+2N_2-2}
 =Ce^{CM_0\e^2}\e^{2\lambda -4} \quad\mbox{for }~ |x|\le R_0\e. \end{equation}
This implies that
\[ \big\|g_{2,\e,\alpha}(\xi,\eta) \big\|_{L^2(|x|\le R_0\e)} \le Ce^{CM_0\e^2}\e^{2\lambda-3}. \]

For $|x|\ge R_0\e$, we express $g_{2,\e}$ as
\begin{equation}
  \label{g2e} \frac{1}{4-ab} g_{2,\e,\alpha}(\xi,\eta)(x) =J_1+J_2+J_3, \end{equation}
where
\begin{align*}
 J_1 &= -\frac{1}{2\e^2} e^{W_\alpha(x)} \Big(e^{H_\e(x) -(b/2)(U+\e^2\varphi_\e +\e^2\xi)(x/\e)
 +\e^2\eta(x)} -1-\e^2\eta(x)\Big),  \\
 J_2 &= -\frac{b}{4\e^2} e^{W_\alpha(x)} \Big( e^{H_\e(x) +\frac{2-b}{2}(U+\e^2\varphi_\e
 +\e^2\xi)(x/\e) +\e^2\eta(x)} -1 -\e^2\eta(x) \Big),  \\
 J_3 &= e^{2W_\alpha(x) +2H_\e(x) -b(U+\e^2\varphi_\e+\e^2\xi)(x/\e) +2\e^2\eta(x)}.
\end{align*}
For simplicity, we let
\begin{align}
 R_{1,\e}(x) &= H_\e(x) -\frac{b}{2}\big(U+\e^2\varphi_\e +\e^2\xi\big)(x/\e) +\e^2\eta(x), \label{eq:R1e} \\
 R_{2,\e}(x) &= H_\e(x) +\frac{2-b}{2} \big(U+\e^2\varphi_\e +\e^2\xi \big)(x/\e) +\e^2\eta(x). \label{eq:R2e}
\end{align}
For $|x|\ge R_0\e$ we can rewrite $J_1$ as
\begin{align*}
 J_1 &= -\frac{1}{2\e^2} e^{W_\alpha(x)} \big(e^{R_{1,\e}(x)} -1 -R_{1,\e}(x)\big) \\
&\quad -\frac{1}{2\e^2} e^{W_\alpha(x)} \Big( H_\e(x)
 -\frac{b}{2}\big(U+\e^2\varphi_\e +\e^2\xi\big)(x/\e) \Big).
\end{align*}
Since $|H_\e(x)|\le C\e^2/|x|^2$ for $|x|\ge R_0\e$, it follows from \eqref{eq:xi-eta-Linfty} that
if $\lambda\ge 2$ then
\begin{align*}
 |J_1| &\le \frac{1}{4\e^2} e^{W_\alpha(x)} e^{|R_{1,\e}(x)|}|R_{1,\e}(x)|^2
 +C e^{W_\alpha(x)} \Big( \frac{1}{|x|^2} +\frac{1}{\e^2}\Big|U \Big(\frac{x}{\e}\Big)\Big|
 +\Big|\xi\Big(\frac{x}{\e}\Big)\Big| +1 \Big) \\
&\le Ce^{CM_0\e^2} (\sigma^{c_0M_0\e^2} e^{W_\alpha})(x) \Big( \frac{1}{|x|^2}
 +\frac{1}{\e^2} e^{-|x|/\e} +(1+M_0\e^2) \Big|\xi\Big(\frac{x}{\e}\Big)\Big| + \\
&\qquad +1 +M_0^2\e^2 +M_0^2\e^2(\ln\sigma)^2(x) \Big)  ~\quad\mbox{for }~ |x|\ge R_0\e.
\end{align*}
If $\lambda=3/2$ then $p_1={\bf 0}$ and hence $H_\e\equiv 0$ as before. In this case
\begin{align*}
 |J_1| &\le Ce^{CM_0\e^2} (\sigma^{c_0M_0\e^2} e^{W_\alpha})(x)
 \Big( \frac{1}{\e^2} e^{-|x|/\e} +(1+M_0\e^2) \Big|\xi\Big(\frac{x}{\e}\Big)\Big| + \\
&\qquad +1 +M_0^2\e^2 +M_0^2\e^2(\ln\sigma)^2(x) \Big)  ~\quad\mbox{for }~ |x|\ge R_0\e.
\end{align*}
Recall that $c_0M_0(\e')^2\le 1$. If $0<\e<{\e}'$ and $|\alpha|\le 1$ then
\[ (\sigma^{c_0M_0\e^2} e^{W_\alpha})(x) \le C|x|^{2\lambda-2}\sigma^{1-4\lambda}(x)
 \quad\mbox{for }~ |x|\ge R_0\e. \]
Consequently, if $\lambda\ge 3/2$, $0<\e<{\e}'$ and $|\alpha|\le 1$ then
\begin{align*}
& \int_{|x|\ge R_0\e} (\sigma^{2+2d}\sigma^{2c_0M_0\e^2} e^{2W_\alpha})(x) e^{-2|x|/\e} dx \\
&\quad\le \int_{|x|\ge R_0\e} C|x|^{4\lambda-4} e^{-2|x|/\e} dx
 \le \int_{|y|\ge R_0} C\e^{4\lambda-2} |y|^{4\lambda-4} e^{-2|y|} dy \le C\e^{4\lambda -2}
\end{align*}
and
\[ \int_{|x|\ge R_0\e} \sigma^{2+2d}\sigma^{2c_0M_0\e^2}e^{2W_\alpha} \Big|\xi\Big(\frac{x}{\e}\Big)\Big|^2dx
 \le C\e^2\|\xi\|_{\LN}^2 \le CM_0^2\e^2. \]
If $\lambda\ge 2$ in addition, then
\[ \int_{|x|\ge R_0\e} \frac{1}{|x|^4} \sigma^{2+2d} (\sigma^{2c_0M_0\e^2}e^{2W_\alpha})(x)dx \le C. \]
Therefore if $0<\e<{\e}'$ and $|\alpha|\le 1$,
\begin{equation}
  \label{eq:J1-Ynorm}
 \|J_1\|_Y\le Ce^{CM_0\e^2}(1 +M_0^2\e^2).
\end{equation}

$J_2$ can be expressed as
\begin{align*}
 J_2 &= -\frac{b}{4\e^2} e^{W_\alpha(x)} \big(e^{R_{2,\e}(x)} -1 -R_{2,\e}(x)\big) \\
&\quad -\frac{b}{4\e^2} e^{W_\alpha(x)} \Big( H_\e(x)
 +\frac{2-b}{2}\big(U+\e^2\varphi_\e +\e^2\xi\big)(x/\e) \Big).
\end{align*}
Similarly, we obtain that $\|J_2\|_Y\le Ce^{CM_0\e^2}(1 +M_0^2\e^2)$
for $0<\e<{\e}'$ and $|\alpha|\le 1$.

Clearly $\|J_3\|_Y\le Ce^{CM_0\e^2}$ for $0<\e<{\e}'$ and $|\alpha|\le 1$.
Combining all these estimates, we obtain \eqref{eq:g2e-defined}.\\

We have proved that if $0<\e<\e''=\min\{{\e}',\e_1\}$ and $|\alpha|<\e''$ then
$\mathcal{L}_1^{-1}g_{1,\e,\alpha}(\xi,\eta)\in H^2(\RN)$ and
$(\mathcal{L}_{2,\alpha})^{-1} T_\alpha g_{2,\e,\alpha}(\xi,\eta)\in E_\alpha$
for all $(\xi,\eta)\in S_0$. Moreover it follows from Theorem \ref{thm:L-isomorphism},
\eqref{eq:g1e-defined} and \eqref{eq:g2e-defined} that there exist constants $C_i=C_i(p_j,q_k)$ such that
\begin{align*}
 \|\mathcal{L}_1^{-1}g_{1,\e,\alpha}(\xi,\eta)\|_{H^2(\RN)}
&\le C_3C_1M_0e^{C_1M_0\e^2}\e, \\
 \|(\mathcal{L}_{2,\alpha})^{-1} T_\alpha g_{2,\e,\alpha}(\xi,\eta)\|_X
&\le C_4C_2e^{C_2M_0\e^2}(1+M_0^2\e^2)
\end{align*}
for all $(\xi,\eta)\in S_0$. We let
\[ M_0 =1+2C_3C_1 +2C_4C_1. \]
Thus there exists a number $\hat{\e}\in (0,\e'')$ such that if $0<\e<\hat{\e}$ and
$|\alpha|<\hat{\e}$ then the map
\[ \Gamma_{\e,\alpha}(\xi,\eta) =\big(\mathcal{L}_1^{-1}g_{1,\e,\alpha}(\xi,\eta),
 ~(\mathcal{L}_{2,\alpha})^{-1}T_\alpha g_{2,\e,\alpha} (\xi,\eta) \big). \]
is a well-defined map from $S_0$ into $S_0$. \\

Now we show that $\Gamma_{\e,\alpha}:S_0\to S_0$ is contractive if $\e>0$ and $|\alpha|$ are sufficiently small.
Let $(\xi_1,\eta_1), (\xi_2,\eta_2)\in S_0$ be given.
For simplicity, we write $g_j=g_{j,\e,\alpha}$ ($j=1,2$).

We first estimate $\|g_{1}(\xi_1,\eta_1) -g_{1}(\xi_2,\eta_2)\|_{\LN}$. Note that
\begin{align*}
& g_{1}(\xi_1,\eta_1)(x) -g_{1}(\xi_2,\eta_2)(x) \\
&\quad = -\frac{1}{\e^2}\Big( f(U+\e^2\varphi_\e+\e^2\xi_1) -f(U+\e^2\varphi_\e+\e^2\xi_2)
 -\e^2 f'(U)(\xi_1-\xi_2)\Big)(x) \\
&\quad\quad +ae^{V_{2,\e,\alpha}(x)+\e^2\eta_2(\e x) -(b/2)\e^2\xi_2(x)}
 \big(e^{\e^2(\eta_1-\eta_2)(\e x) -(b/2)\e^2(\xi_1-\xi_2)(x)} -1\big) \\
&\quad\quad +a(b-2)e^{(V_{1,\e}+V_{2,\e,\alpha})(x) +\frac{2-b}{2}\e^2\xi_2(x) +\e^2\eta_2(\e x)}
 \Big( e^{\frac{2-b}{2}\e^2(\xi_1-\xi_2)(x) +\e^2(\eta_1-\eta_2)(\e x)} -1\Big) \\
&\quad\quad -2a\e^2e^{2V_{2,\e,\alpha}(x) -b\e^2\xi_2(x) +2\e^2\eta_2(\e x)}
 \big( e^{-b\e^2(\xi_1-\xi_2)(x) +2\e^2(\eta_1-\eta_2)(\e x)} -1 \big).
\end{align*}
It is easily verified that
\begin{align*}
 I_1^* &:= -\frac{1}{\e^2}\Big( f(U+\e^2\varphi_\e+\e^2\xi_1) -f(U+\e^2\varphi_\e+\e^2\xi_2)
 -\e^2 f'(U)(\xi_1-\xi_2)\Big) \\
&= e^U(1-e^{\e^2\varphi_\e +\e^2\xi_2})(\xi_1-\xi_2)
 +2e^{2U}(e^{2\e^2\varphi_\e +2\e^2\xi_2}-1)(\xi_1-\xi_2) \\
&\quad -\frac{1}{\e^2} e^{U+\e^2\varphi_\e+\e^2\xi_2}\big( e^{\e^2(\xi_1-\xi_2)}-1-\e^2(\xi_1-\xi_2)\big) \\
&\quad
+\frac{1}{\e^2} e^{2U+2\e^2\varphi_\e+2\e^2\xi_2}\big( e^{2\e^2(\xi_1-\xi_2)}-1-2\e^2(\xi_1-\xi_2)\big),
\end{align*}
and consequently
\[ |I_1^*| \le C\e^2e^{CM_0\e^2}|\xi_1-\xi_2|^2 +CM_0\e^2e^{CM_0\e^2}|\xi_1-\xi_2|. \]
Then it follows that
\[ \|I_1^*\|_{\LN} \le CM_0\e^2e^{CM_0\e^2}\|\xi_1-\xi_2\|_{H^2(\RN)}. \]

If we let
\[ I_2^*(x)= e^{V_{2,\e,\alpha}(x)+\e^2\eta_2(\e x) -(b/2)\e^2\xi_2(x)}
 \big(e^{\e^2(\eta_1-\eta_2)(\e x) -(b/2)\e^2(\xi_1-\xi_2)(x)} -1\big), \]
then
\begin{align*}
 |I_2^*(x)| &\le C\e^2e^{CM_0\e^2} \sigma^{3c_0M_0\e^2}(\e x) e^{V_{2,\e,\alpha}(x)}
 \big(|\xi_1-\xi_2|(x) +|\eta_1-\eta_2|(\e x) \big) 
\end{align*}
Recall that $e^{V_{2,\e,\alpha}(x)}\le C\e^{2\lambda -2}$ for $|x|\le R_0$,
and $e^{V_{2,\e,\alpha}(x)}\le Ce^{W_\alpha(\e x)}$ for $|x|\ge R_0$.
Thus if $3c_0M_0\e^2\le 1$ then it follows from \eqref{eq:xi-eta-Linfty} that
\begin{align*}
 \|I_2^*\|_{\LN} \le C\e e^{CM_0\e^2} \big(\|\xi_1-\xi_2\|_{H^2(\RN)} +\|{\eta}_1-{\eta}_2\|_X \big).
\end{align*}
Repeating the above estimates to the remaining two quantities, we conclude that
\begin{equation}
  \label{eq:g1e-difference}
 \|g_{1}(\xi_1,\eta_1) -g_{1}(\xi_2,\eta_2)\|_{\LN}
 \le CM_0\e e^{CM_0\e^2} \big( \|\xi_1-\xi_2\|_{H^2(\RN)} +\|\eta_1-\eta_2\|_X \big)
\end{equation}
if $3c_0M_0\e^2\le 1$ and $|\alpha|\le 1$.

We now estimate $\|g_{2}(\xi_1,\eta_1) -g_{2}(\xi_2,\eta_2)\|_Y$.
It is easily checked that
\[ \frac{1}{4-ab}\big(g_{2}(\xi_1,\eta_1) -g_{2}(\xi_2,\eta_2)\big)(x)
 =J_1^*(x) +J_2^*(x) +J_3^*(x), \]
where
\begin{align*}
 J_1^*(x) &= -\frac{1}{2\e^2} e^{{V}_{2,\e,\alpha}(x/\e)} \big( e^{\e^2{\eta}_1(x) -(b/2)\e^2{\xi}_1(x/\e)}
 -e^{\e^2{\eta}_2(x) -(b/2)\e^2{\xi}_2(x/\e)}\big) \\
&\quad +\frac12 e^{W_\alpha(x)}({\eta}_1 -{\eta}_2)(x), \\
 J_2^*(x) &= -\frac{b}{4\e^2} e^{({V}_{2,\e,\alpha} +{U} +\e^2{\varphi}_\e)(x/\e)}
 \big( e^{\e^2{\eta}_1(x) +\frac{2-b}{2}\e^2{\xi}_1(x/\e)}
 -e^{\e^2{\eta}_2(x) +\frac{2-b}{2}\e^2{\xi}_2(x/\e)}\big) \\
&\quad +\frac{b}{4} e^{W_\alpha(x)}({\eta}_1 -{\eta}_2)(x) \\
 J_3^*(x) &= e^{2{V}_{2,\e,\alpha}(x/\e)} \big( e^{2\e^2{\eta}_1(x) -b\e^2{\xi}_1(x/\e)}
 -e^{2\e^2{\eta}_2(x) -b\e^2{\xi}_2(x/\e)}\big).
\end{align*}
It follows from the mean value theorem that if $|x|\le R_0\e$ then
\[ \big| g_{2}(\xi_1,\eta_1)(x) -g_{2}(\xi_2,\eta_2)(x)\big|
 \le Ce^{CM_0\e^2}\e^{2\lambda -2} \big(|{\xi}_1-{\xi}_2|(x/\e) +|{\eta}_1-{\eta}_2|(x) \big). \]

We estimate $|J_1^*|$, $|J_2^*|$ and $|J_3^*|$ for $|x|\ge R_0\e$. If we write
\[ e^{V_{2,\e,\alpha}(x/\e)} =e^{W_\alpha(x)} \big( e^{H_\e(x) -(b/2)({U}+\e^2{\varphi}_\e)(x/\e)} -1\big)
 +e^{W_\alpha(x)} \]
and
\begin{align*}
& e^{\e^2{\eta}_1(x) -(b/2)\e^2{\xi}_1(x/\e)} -e^{\e^2{\eta}_2(x) -(b/2)\e^2{\xi}_2(x/\e)} \\
&\quad = e^{\e^2\eta_1(x)} -e^{\e^2\eta_2(x)} +(e^{-(b/2)\e^2\xi_1(x/\e)} -1)(e^{\e^2\eta_1}-e^{\e^2\eta_2})(x) \\
&\qquad +e^{\e^2\eta_2(x)} (e^{-(b/2)\e^2\xi_1} -e^{-(b/2)\e^2\xi_2})(x/\e)
 \quad\mbox{for }~ |x|\ge R_0\e,
\end{align*}
then the mean value theorem implies that if $|x|\ge R_0\e$ then
\begin{align*}
 |J_1^*| &\le Ce^{CM_0\e^2}(\sigma^{3c_0M_0\e^2}e^{W_\alpha})(x) \big( \e^2|x|^{-2} +e^{-|x|/\e} +\e^2\big)
 |{\xi}_1-{\xi}_2|(x/\e) \\
&\quad +Ce^{CM_0\e^2}(\sigma^{3c_0M_0\e^2}e^{W_\alpha})(x) \big( \e^2|x|^{-2} +e^{-|x|/\e} +\e^2\big)
 |{\eta}_1-{\eta}_2|(x) \\
&\quad +Ce^{CM_0\e^2} (\sigma^{3c_0M_0\e^2}e^{W_\alpha})(x) \big( |{\xi}_1-{\xi}_2|(x/\e)
 + \e^2({\eta}_1-{\eta}_2)^2(x)  \big).
\end{align*}
Since $\lambda\ge 3/2$, if $\e>0$ is sufficiently small and $|\alpha|\le 1$ then
\begin{align*}
 \|J_1^*\|_Y \le CM_0\e e^{CM_0\e^2} \big( \|\xi_1-\xi_2\|_{H^2(\RN)} +\|{\eta}_1 -{\eta}_2\|_X\big).
\end{align*}
Similarly, we obtain that
$\|J_2^*\|_Y \le CM_0\e e^{CM_0\e^2} \big( \|\xi_1-\xi_2\|_{H^2(\RN)} +\|{\eta}_1 -{\eta}_2\|_X\big)$.
Finally it follows from the mean value theorem that
\[ \|J_3^*\|_Y \le C\e^2e^{CM_0\e^2} \big( \|\xi_1-\xi_2\|_{H^2(\RN)} +\|{\eta}_1-{\eta}_2\|_X\big). \]
From all these estimates, it follows that
\begin{equation}
  \label{eq:g2e-difference}
 \| g_{2}(\xi_1,\eta_1) -g_{2}(\xi_2,\eta_2)\|_Y
 \le CM_0\e e^{CM_0\e^2} \big( \|\xi_1-\xi_2\|_{H^2(\RN)} +\|\eta_1-\eta_2\|_X\big)
\end{equation}
if $\e>0$ is sufficiently small and $|\alpha|\le 1$.

Therefore we can choose a constant $\overline{\e}_2\in (0,\hat{\e})$ such that
$\Gamma_{\e,\alpha}:S_0\to S_0$ is a well-defined contraction map provided that $0<\e<\overline{\e}_2$
and $|\alpha|<\overline{\e}_2$. The contraction mapping theorem implies that,
if $0<\e<\overline{\e}_2$ and $|\alpha|<\overline{\e}_2$ then
$\Gamma_{\e,\alpha}$ has a unique fixed point in $S_0$. This proves Proposition \ref{prop:contraction}.
\end{proof}

By Proposition \ref{prop:contraction}, if $0<\e<\overline{\e}_2$ and $|\alpha|<\overline{\e}_2$
then $(\xi_{\e,\alpha},\eta_{\e,\alpha})\in S_0$ satisfies
\begin{align*}
 \mathcal{L}_1\xi_{\e,\alpha} - g_{1,\e,\alpha}(\xi_{\e,\alpha},\eta_{\e,\alpha}) &=0, \\
 T_\alpha\big( \mathcal{L}_{2,\alpha}\eta_{\e,\alpha} -g_{2,\e,\alpha}(\xi_{\e,\alpha},
 \eta_{\e,\alpha}) \big) &=0.
\end{align*}
Here we used $T_\alpha\mathcal{L}_{2,\alpha}=\mathcal{L}_{2,\alpha}$ on $E_\alpha$.
Moreover $\|\xi_{\e,\alpha}\|_{H^2(\RN)}\le C\e$ and $\|\eta_{\e,\alpha}\|_X\le C$ for some
constant $C$ independent of $\e$ and $\alpha$ as $(\e,\alpha)\to (0,{\bf 0})$.

We claim the map $(\e,\alpha)\mapsto (\xi_{\e,\alpha},\eta_{\e,\alpha})$ is continuous.
Indeed, there holds
\[ |g_{j,\e_2,\alpha_2}(\xi_{\e_2,\alpha_2},\eta_{\e_2,\alpha_2})
 -g_{j,\e_1,\alpha_1}(\xi_{\e_1,\alpha_1},\eta_{\e_1,\alpha_1}) | \le \triangle_{1,j} +\triangle_{2,j}, \]
where
\begin{align*}
 \triangle_{1,j} &= |g_{j,\e_2,\alpha_2}(\xi_{\e_2,\alpha_2},\eta_{\e_2,\alpha_2})
 -g_{j,\e_1,\alpha_1}(\xi_{\e_2,\alpha_2},\eta_{\e_2,\alpha_2})|, \\
 \triangle_{2,j} &= |g_{j,\e_1,\alpha_1}(\xi_{\e_2,\alpha_2},\eta_{\e_2,\alpha_2})
 -g_{j,\e_1,\alpha_1}(\xi_{\e_1,\alpha_1},\eta_{\e_1,\alpha_1})|.
\end{align*}
Since $(\xi_{\e_j,\alpha_j},\eta_{\e_j,\alpha_j})\in S_0$, the Lebesgue convergence theorem implies that
$\|\triangle_{1,j}\|_{\LN}=o(1)$ as $(\e_2,\alpha_2)\to (\e_1,\alpha_1)$.
It follows from the proof of \eqref{eq:g1e-difference} that
$\|\triangle_{2,j}\|_{\LN}=o(1)$ as $(\e_2,\alpha_2)\to (\e_1,\alpha_1)$.
Then Theorem \ref{thm:L-isomorphism} implies that
\[ \|\xi_{\e_2,\alpha_2}-\xi_{\e_1,\alpha_1}\|_{H^2(\RN)} +\|\eta_{\e_2,\alpha_2}-\eta_{\e_1,\alpha_1}\|_{X}
 \to 0 \quad\mbox{as }~ (\e_2,\alpha_2)\to (\e_1,\alpha_1). \]
This proves the claim. We skip the details. \\

Recall that $\lambda\ge 3/2$.
If $\lambda\notin\mathbb{N}$ then $T_\alpha:Y\to Y$ is an identity.
In this case $(\xi_{\e,\alpha},\eta_{\e,\alpha})$ is a solution of the system
\eqref{eq:eq-L1}-\eqref{eq:eq-L2}, and hence Theorem \ref{T121} is proved when $\lambda\notin\mathbb{N}$.

If $\lambda\in\mathbb{N}$ there exist constants $c_{1,\e,\alpha},c_{2,\e,\alpha}\in\mathbb{R}$ such that
\[ \mathcal{L}_{2,\alpha}\eta_{\e,\alpha} -g_{2,\e,\alpha}(\xi_{\e,\alpha},\eta_{\e,\alpha})
 = c_{1,\e,\alpha}\sigma^{-2-2d}Z_{\alpha,1} +c_{2,\e,\alpha}\sigma^{-2-2d}Z_{\alpha,2}, \]
 and \begin{equation*}
   \int_{\RN} \Big(\mathcal{L}_{2,\alpha } \eta_{\e,\alpha }
 -g_{2,\e,\alpha }(\xi_{\e,\alpha },\eta_{\e,\alpha })-\sum_{i=1}^2c_{i,\e,\alpha}\sigma^{-2-2d}Z_{\alpha,i} \Big) Z_{\alpha(\e),j} dx =0\ \ (j=1,2),
\end{equation*}
for any $\e\in (0,\overline{\e}_2)$ and $|\alpha|<\overline{\e}_2$.
To complete the proof of Theorem \ref{T121} for $\lambda\ge 3/2$, in the following proposition we will prove that if $\e>0$ is sufficiently small and
the singular points $p_j,q_k$ satisfy some conditions then there exists an $\alpha(\e)\in\mathbb{C}$
such that $c_{1,\e,\alpha(\e)}=c_{2,\e,\alpha(\e)}=0$.

\begin{proposition}
  \label{prop:asymptotic}
Suppose $\lambda\in\mathbb{N}$ and one of the following conditions holds.
\begin{itemize}
 \item[{\rm (i)}] $\lambda\ge 3$.
 \item[{\rm (ii)}] $\lambda=2$ and $p_j=q_k={\bf 0}$ for all $j,k$.
\end{itemize}
Then there exists a constant $\e_*\in (0,\overline{\e}_2)$ satisfying the following property:
for each $0<\e<\e_*$ there exists an $\alpha=\alpha(\e)\in \mathbb{C}$ such that
\begin{equation}
  \label{eq:asymptotic}
 \int_{\RN} \Big(\mathcal{L}_{2,\alpha(\e)} \eta_{\e,\alpha(\e)}
 -g_{2,\e,\alpha(\e)}(\xi_{\e,\alpha(\e)},\eta_{\e,\alpha(\e)})\Big) Z_{\alpha(\e),j} dx =0,
\end{equation}and $c_{j,\e,\alpha(\e)}\equiv0$,  $(j=1,2)$.
Moreover, $|\alpha(\e)|\le C\e$ as $\e\to 0$.

\end{proposition}
\begin{proof}We remark that   the proof of Lemma \ref{lem:T-aux} yields that \eqref{eq:asymptotic} implies  $c_{j,\e,\alpha(\e)}\equiv0$,  $(j=1,2)$. So we are going to prove \eqref{eq:asymptotic}. \\
Since $\eta_{\e,\alpha}\in X$ and $\mathcal{L}_{2,\alpha}Z_{\alpha,j}=0$, it follows that
\[ \int_{\RN} Z_{\alpha,j}\mathcal{L}_{2,\alpha} \eta_{\e,\alpha} dx
 =\int_{\RN} \eta_{\e,\alpha} \mathcal{L}_{2,\alpha}Z_{\alpha,j} dx =0, \quad (j=1,2) \]
which in turn implies that
\[ \int_{\RN} \Big(\mathcal{L}_{2,\alpha} \eta_{\e,\alpha}
 -g_{2,\e,\alpha}(\xi_{\e,\alpha}, \eta_{\e,\alpha})\Big) Z_{\alpha,j} dx
 = -\int_{\RN} g_{2,\e}(\cdot\,\xi_{\e,\alpha}, \eta_{\e,\alpha}) Z_{\alpha,j} dx. \]

Let
\begin{equation}
  \label{eq:A}
 A(x)= \frac{1}{|x|^2} \Big( \frac{b}{2}\sum_{j=1}^{N_1} |p_j|^2 +\sum_{k=1}^{N_2} |q_k|^2\Big)
 -\frac{1}{|x|^4} \Big(\sum_{j=1}^{N_1} b(p_j\cdot x)^2 +\sum_{k=1}^{N_2} 2(q_k\cdot x)^2\Big).
\end{equation}
Let $x=(|x|\cos\theta,|x|\sin\theta)$. Then we see that
\begin{equation}\label{cossin}A(x)=A_1\frac{\cos 2\theta}{|x|^2}+A_2\frac{\sin 2\theta}{|x|^2}\ \ \textrm{for some constant}\ \ A_1, A_2\in \mathbb{R}. \end{equation}

We claim that there exists a constant $C=C(p_j,q_k)$ such that
\begin{equation}
  \label{eq:He-A}
 |H_\e(x) -\e^2A(x)| \le \frac{C\e^3}{|x|^3} \quad\mbox{for }~ |x|\ge R_0\e,
\end{equation}
where $R_0$ and $H_\e$ are defined in \eqref{eq:R0} and \eqref{eq:He}, respectively.
To prove \eqref{eq:He-A}, we let $\displaystyle\Phi_j= -\frac{2\e p_j\cdot x}{|x|^2} +\frac{\e^2|p_j|^2}{|x|^2}$
and $\displaystyle\Psi_k= -\frac{2\e q_k\cdot x}{|x|^2} +\frac{\e^2|q_k|^2}{|x|^2}$ for simplicity.
It follows from \eqref{eq:PQ-location} that
\begin{align}
& \sum_{j=1}^{N_1} \frac{b}{2}\Big(\Phi_j -\frac12\Phi_j^2\Big)
 +\sum_{k=1}^{N_2} \Big(\Psi_k -\frac12 \Psi_k^2\Big) \label{eq:claim:He-A-1} \\
&\quad =\e^2A(x) +\sum_{j=1}^{N_1} \Big( \frac{b\e^3|p_j|^2(p_j\cdot x)}{|x|^4} -\frac{b\e^4|p_j|^4}{4|x|^4}\Big)
 +\sum_{k=1}^{N_2} \Big(\frac{2\e^3|q_k|^2(q_k\cdot x)}{|x|^4} -\frac{\e^4|q_k|^4}{2|x|^4}\Big). \nonumber
\end{align}
We also note that
\[ |\Phi_j| \le\frac{2\e|p_j|}{|x|} +\frac{\e^2|p_j|^2}{|x|^2} \le\frac{2|p_j|}{R_0} +\frac{|p_j|^2}{R_0^2}
 \le\frac12 \quad\mbox{for }~ |x|\ge R_0\e. \]
Similarly, $|\Psi_k|\le 1/2$ for $|x|\ge R_0\e$.
Since $|\ln(1+t)-t+(t^2/2)|\le 3|t|^3$ for $|t|\le 1/2$, it follows that
\[ \Big|H_\e(x) -\sum_{j=1}^{N_1} \frac{b}{2}\Big(\Phi_j -\frac12\Phi_j^2\Big)
 -\sum_{k=1}^{N_2} \Big(\Psi_k -\frac12 \Psi_k^2\Big)\Big| \le \frac{C\e^3}{|x|^3}
 \quad\mbox{for }~ |x|\ge R_0\e. \]
Then \eqref{eq:claim:He-A-1} proves the claim \eqref{eq:He-A}.

For convenience, we write
\[ Z_\alpha(z) =\frac{z^\lambda+\alpha}{1+|z^\lambda+\alpha|^2}, ~\quad z=x_1+{\rm i}x_2, \]
so that $Z_\alpha=Z_{\alpha,1}+{\rm i}Z_{\alpha,2}$.
We now consider two cases separately. \\

\noindent {\it Case (i)}. Suppose that $\lambda\ge 3$.

We claim that if $|\alpha|<\overline{\e}_2$ then
\begin{equation}
  \label{eq:g2e-Z-delta}
 \int_{\RN} g_{2,\e,\alpha} (\xi_{\e,\alpha},\eta_{\e,\alpha}) Z_\alpha dx
 =\triangle(\alpha)+O(\e) \quad\mbox{uniformly as }~ \e\to 0,
\end{equation}
where we set
\begin{align*}
 \triangle(\alpha) &= -\frac14 (2+b)(4-ab) \int_{\RN} e^{W_\alpha}AZ_\alpha dx
 +\frac{1}{16} (16-ab^3)(4-ab) \int_{\RN} e^{2W_\alpha}Z_\alpha dx \\
&= -8\lambda^2 \int_{\RN} \frac{|z|^{2\lambda-2}(z^\lambda+\alpha)A(z)}{ (1+|z^\lambda+\alpha|^2)^3} dx
 +\frac{64\lambda^4(16-ab^3)}{(4-ab)(2+b)^2}
 \int_{\RN} \frac{|z|^{4\lambda-4}(z^\lambda+\alpha)}{(1+|z^\lambda+\alpha|^2)^5} dx.
\end{align*}
%
%
Indeed, we first note that $|g_{2,\e,\alpha}(\xi_{\e,\alpha},\eta_{\e,\alpha})(x)|\le C\e^{2\lambda-4}$
for $|x|\le R_0\e$, and hence
\[ \int_{|x|\le R_0\e} \big| g_{2,\e,\alpha}(\xi_{\e,\alpha},\eta_{\e,\alpha}) Z_{\alpha}\big|dx
 \le C\e^{2\lambda-2} \le C\e. ~\quad (j=1,2) \]

If $|x|\ge R_0\e$,  by \eqref{g2e} and \eqref{eq:varphi-e},  it is easily verified that
\begin{equation}
  \label{eq:g2e-decomp}
 \frac{1}{4-ab} g_{2,\e,\alpha} (\xi_{\e,\alpha},\eta_{\e,\alpha})
 = -\frac{2+b}{4} e^{W_\alpha}A +\frac{16-ab^3}{16} e^{2W_\alpha} +M_\e,
\end{equation}
where
\begin{align*}
M_\e &= -\frac{1}{2\e^2} e^{W_\alpha}(e^{R_{1,\e}}-1-R_{1,\e})
 -\frac{b}{4\e^2}e^{W_\alpha}(e^{R_{2,\e}}-1-R_{2,\e}) \\
&\quad +e^{2W_\alpha}(e^{2R_{1,\e}}-1) +\frac{b^2}{8\e^2}e^{W_\alpha} \big({U}(\cdot/\e)
+\e^2{\xi}_{\e,\alpha}(\cdot/\e)\big) -\frac{2+b}{4\e^2} e^{W_\alpha}(H_\e -\e^2A).
\end{align*}
Here the functions $R_{1,\e}$ and $R_{2,\e}$ are given in \eqref{eq:R1e}-\eqref{eq:R2e}
with $\xi=\xi_{\e,\alpha}$ and $\eta=\eta_{\e,\alpha}$.

Since $\lambda\ge 3$, it follows that
\begin{align*}
& \Big|\int_{|x|\ge R_0\e} \frac{1}{\e^2} e^{W_\alpha}(e^{R_{j,\e}}-1-R_{j,\e})Z_\alpha dx \Big|
\le \int_{|x|\ge R_0\e} \frac{C}{\e^2} |x|^{2\lambda-2}\sigma^{-4\lambda} |R_{j,\e}|^2 dx \\
&~ \le \int_{|x|\ge R_0\e} \frac{C}{\e^2} |x|^{2\lambda-2}\sigma^{-4\lambda}
 \Big(\frac{\e^4}{|x|^4} +e^{-2|x|/\e} +\e^4 +\e^4(\ln\sigma)^2 \Big) dx \le C\e
\end{align*}
and
\[ \Big|\int_{|x|\ge R_0\e} e^{2W_\alpha}(e^{2R_{1,\e}}-1)Z_\alpha dx\Big|
 \le C\e^{4\lambda-2}+C\e^2 \le C\e^2. \]
Since $\|\xi_{\e,\alpha}\|_{\LN}\le C\e$, we obtain that
\begin{align*}
& \Big| \frac{1}{\e^2} \int_{|x|\ge R_0\e} e^{W_\alpha(x)} ({U}+\e^2\xi_{\e,\alpha})(x/\e)
 Z_\alpha(x)dx \Big| \\
&\quad \le \frac{C}{\e^2} \int_{|x|\ge R_0\e} |x|^{2\lambda-2}e^{-|x|/\e} dx
 +C\|\xi_{\e,\alpha}(\cdot/\e)\|_{L^2(|x|\ge R_0\e)} \\
&\quad \le C\e^{2\lambda-2} +C\e\|\xi_{\e,\alpha}\|_{\LN} ~\le~ C\e^{2\lambda-2} +C\e^2 \le C\e.
\end{align*}
Finally it follows from \eqref{eq:He-A} that
\begin{align*}
 \Big| \int_{|x|\ge R_0\e} \frac{1}{\e^2} e^{W_\alpha}(H_\e-\e^2A)Z_\alpha dx\Big|
 \le C\e\int_{|x|\ge R_0\e} |x|^{2\lambda-5}\sigma^{-5\lambda} dx \le C\e.
\end{align*}
Then our claim \eqref{eq:g2e-Z-delta} follows from \eqref{eq:g2e-decomp} and the above error estimates.

We claim that
\begin{equation}
  \label{eq:Delta-alpha}
 \triangle(\alpha) = \frac{64\lambda^4(ab^3-16)(\lambda-1)\pi\alpha}{(4-ab)(2+b)^2\lambda}
 \int_0^\infty \frac{t^{2\lambda-2}}{(1+t^\lambda)^5} dt +O(|\alpha|^2)
 \quad\mbox{as }~ |\alpha|\to 0.
\end{equation}
Indeed, we note that as $|\alpha|\to 0$,
\begin{align*}
 \frac{z^\lambda+\alpha}{(1+|z^\lambda+\alpha|^2)^3} &= \frac{z^\lambda}{(1+|z|^{2\lambda})^3}
 +\frac{\alpha(1-2|z|^{2\lambda})-3\overline{\alpha}z^{2\lambda}}{(1+|z|^{2\lambda})^4}
 +O\big(|\alpha|^2\sigma^{-7\lambda}\big), \\
 \frac{z^\lambda+\alpha}{(1+|z^\lambda+\alpha|^2)^5} &= \frac{z^\lambda}{(1+|z|^{2\lambda})^5}
 +\frac{\alpha(1-4|z|^{2\lambda})-5\overline{\alpha}z^{2\lambda}}{(1+|z|^{2\lambda})^6}
 +O\big(|\alpha|^2\sigma^{-11\lambda}\big),
\end{align*}
where $\overline{\alpha}$ denotes the complex conjugate of $\alpha$.

If we introduce the polar coordinates $x=(r\cos\theta, r\sin\theta)$ then we obtain from \eqref{cossin} and $\lambda\ge3$ that
\begin{align*}
 \int_{\RN} \frac{|z|^{2\lambda-2}z^\lambda A(z)}{(1+|z|^{2\lambda})^3}dz=0,\ \ \int_{\RN} \frac{|z|^{4\lambda-4}z^\lambda }{(1+|z|^{2\lambda})^5}dz=0, \ \ \textrm{and thus}
\end{align*}
\begin{align*}
 \triangle(\alpha) &= -8\lambda^2\alpha
 \int_{\RN} \frac{|x|^{2\lambda-2}(1-2|x|^{2\lambda})A(x)}{ (1+|x|^{2\lambda})^4} dx \\
&\quad +\frac{64\lambda^4(16-ab^3)\alpha}{(4-ab)(2+b)^2}
 \int_{\RN} \frac{|x|^{4\lambda-4}(1-4|x|^{2\lambda})}{(1+|x|^{2\lambda})^6} dx +O(|\alpha|^2)
 \quad\mbox{as }~ |\alpha|\to 0.
\end{align*}
Here we used $\lambda\ge 3$. We also obtain that
\[ \int_{\RN} \frac{|x|^{2\lambda-2}(1-2|x|^{2\lambda}) A(x)}{ (1+|x|^{2\lambda})^4} dx =0. \]
%
%
Moreover, integration by parts (\cite{CI}) yields
\begin{align*}
& \int_{\RN} \frac{|x|^{4\lambda-4}(1-4|x|^{2\lambda})}{(1+|x|^{2\lambda})^6} dx \\
&\quad = 2\pi\int_0^\infty \Big(\frac{5r^{4\lambda-4}}{(1+r^{2\lambda})^6}
 -\frac{4r^{4\lambda-4}}{(1+r^{2\lambda})^5}\Big) rdr
 = \pi\int_0^\infty \Big(\frac{5t^{2\lambda-2}}{(1+t^{\lambda})^6}
 -\frac{4t^{2\lambda-2}}{(1+t^{\lambda})^5}\Big) dt \\
&\quad = \frac{\pi}{\lambda} \Big[-\frac{t^{\lambda-1}}{(1+t^\lambda)^5} +\frac{t^{\lambda-1}}{(1+t^\lambda)^4}
 \Big]_0^\infty +\frac{(\lambda-1)\pi}{\lambda} \int_0^\infty \Big( \frac{t^{\lambda-2}}{(1+t^\lambda)^5}
 -\frac{t^{\lambda-2}}{(1+t^\lambda)^4} \Big) dt \\
&\quad =-\frac{(\lambda-1)\pi}{\lambda} \int_0^\infty \frac{t^{2\lambda-2}}{(1+t^\lambda)^5}dt.
\end{align*}
This proves the claim \eqref{eq:Delta-alpha}. We have proved that, as $\e\to 0$ and $|\alpha|\to 0$,
\begin{align*}
& \int_{\RN} g_{2,\e,\alpha}(\xi_{\e,\alpha},\eta_{\e,\alpha}) Z_\alpha dx \\
&\quad =\frac{64\lambda^4(ab^3-16)(\lambda-1)\pi\alpha}{(4-ab)(2+b)^2\lambda}
 \int_0^\infty \frac{t^{2\lambda-2}}{(1+t^\lambda)^5} dt +O(|\alpha|^2) +O(\e).
\end{align*}

Since $\lambda>1$ and the map $(\e,\alpha)\mapsto \int_{\RN} g_{2,\e,\alpha}(\xi_{\e,\alpha},
 \eta_{\e,\alpha})Z_\alpha dx$ is continuous,
it follows from the Brouwer fixed point theorem that
there exists a constant $\e_*\in (0,\overline{\e}_2)$ satisfying the following property:
for each $0<\e<\e_*$, there exists an $\alpha(\e)\in\mathbb{C}$ such that
\[ \int_{\RN} g_{2,\e,\alpha(\e)}(\xi_{\e,\alpha(\e)},\eta_{\e,\alpha(\e)}) Z_{\alpha(\e),j} dx =0.
 \quad (j=1,2) \]
It is obvious that $|\alpha(\e)|\le C\e$ as $\e\to 0$. \\

\noindent{\it Case (ii)}. Suppose that $\lambda=2$ and $p_j=q_k={\bf 0}$ for all $j,k$.

If $p_j=q_k={\bf 0}$ for all $j,k$ then $H_\e =A=0$ identically.
In this case, it is easily checked that all the estimates in Case (i) are still valid.
This proves Proposition \ref{prop:asymptotic}.
\end{proof}

We now deal with the remaining case of this paper.

\subsection{The case $\lambda=1$}

In this case $N_1=N_2=0$.
We look for a radially symmetric solution $(u_1,u_2)$ of the form
\begin{align*}
 u_1(r) &= -\ln 2 +\e\xi_\e(r), \\
 u_2(r) &= W_0(\e r) +2\ln\e -\frac{b}{2}\e\xi_\e(r) +\e\eta_\e(\e r). ~\quad (r=|x|)
\end{align*}
In this case, $W_0=W_0^*$ and $e^{W_0}\le C\sigma^{-4}$.
We denote by $H^2_r(\RN)$ the set of radially symmetric functions in $H^2(\RN)$.
$L^2_r(\RN)$, $X_r$ and $Y_r$ are similarly defined.

Then the system \eqref{eq:main-u1}-\eqref{eq:main-u2} can be rewritten as
\[ \mathcal{L}_1\xi_\e =h_{1,\e}(\xi_\e,\eta_\e) \quad\mbox{and}\quad
 \mathcal{L}_2\eta_\e =h_{2,\e}(\xi_\e,\eta_\e), \]
where $\mathcal{L}_1:H_r^2(\RN)\to L_r^2(\RN)$ and $\mathcal{L}_2:X_r\to Y_r$ are defined by
\[ \mathcal{L}_1 =\Delta -1, \qquad \mathcal{L}_2=\Delta +\frac14 (4-ab)(2+b)e^{W_0}, \]
and $h_{1,\e}$ and $h_{2,\e}$ are defined by
\begin{align*}
 h_{1,\e}(\xi,\eta)(r) &= \frac{1}{\e}\big( e^{2\e\xi}-e^{\e\xi}-\e\xi\big)(r)
 +a\e e^{W_0(\e r) -\frac{b}{2}\e\xi(r)+\e\eta(\e r)} \\
&\quad +\frac12 a(b-2) \e e^{W_0(\e r)+\frac{2-b}{2}\e\xi(r) +\e\eta(\e r)}
 -2a\e^3 e^{2W_0(\e r) -b\e\xi(r) +2\e\eta(\e r)}, \\
 h_{2,\e}(\xi,\eta)(r) &= \frac{1}{2\e}(ab-4) e^{W_0(r)}
 \big( e^{-\frac{b}{2}\e\xi(r/\e)+\e\eta(r)} -1-\e\eta(r)\big) \\
&\quad +\frac{b}{4\e}(ab-4) e^{W_0(r)} \big(e^{\frac{2-b}{2}\e\xi(r/\e)+\e\eta(r)} -1-\e\eta(r)\big) \\
&\quad -(ab-4)\e e^{2W_0(r) -b\e\xi(r/\e)+2\e\eta(r)}.
\end{align*}
It is well known that $\mathcal{L}_1$ is a continuous bijection from $H_r^2(\RN)$ onto $L_r^2(\RN)$,
and its inverse is also continuous. Moreover $\mbox{ker}\,\mathcal{L}_2=\mbox{span}\{Z_{0,0}\}$, and
the range of $\mathcal{L}_2$ is $Y_r$.
If we let $E_0^r=\{\xi=\xi(r)\mid (\xi,e^{W_0}Z_{0,0})_{L^2(\RN)}=0\},$ then $\mathcal{L}_2$ is an isomorphism from $E_0^r$
onto $Y_r$.

Let
\[ S_1=\{(\xi,\eta)\in H_r^2(\RN)\times E_0^r\mid \|\xi\|_{H^2(\RN)} +\|\eta\|_X\le M_1\}, \]
where $M_1\ge 1$ is a constant to be defined later.

If $(\xi,\eta)\in S_1$ then
\begin{align*}
 |h_{1,\e}(\xi,\eta)(r)| &\le C\e e^{2\e|\xi|} |\xi|^2
 +C\e e^{CM_1\e} (\sigma^{c_0M_1\e}e^{W_0})(\e r) \\
&\quad +C\e^3e^{CM_1\e} (\sigma^{c_0M_1\e} e^{2W_0})(\e r),\ \ \textrm{and}
\end{align*}
\begin{align*}
 |h_{2,\e}(\xi,\eta)(r)|
&\le Ce^{CM_1\e} (1+M_1\e)(\sigma^{c_0M_1\e}e^{W_0})(r) \big|\xi(r/\e)\big| \\
&\quad +C\e e^{CM_1\e} (\sigma^{2c_0M_1\e}e^{2W_0})(r) \\
&\quad +C\e M_1^2e^{CM_1\e} (\sigma^{c_0M_1\e} e^{W_0})(r) \big(1+(\ln\sigma)^2(r) \big)
\end{align*}
for some constants $C,c_0\ge 1$ independent of $\e$ and $M_1$.
Then we choose a number $\e'=\e'(M_1)>0$ such that $c_0M_1\e'\le 1/3$. Consequently if $0<\e<\e'$ then
\begin{align*}
 \|h_{1,\e}(\xi,\eta)\|_{\LN} &\le C_1e^{C_1M_1\e}(1+M_1^2\e), \\
 \|h_{2,\e}(\xi,\eta)\|_Y &\le C_1e^{C_1M_1\e}(\e +M_1^2\e)
\end{align*}
for some constant $C_1>0$ independent of $\e$ and $M_1$.

Moreover if $(\xi_1,\eta_1), (\xi_2,\eta_2)\in S_1$ and $\e\in (0,\e')$ is sufficiently small then
\begin{align*}
& \|h_{1,\e}(\xi_1,\eta_1)-h_{1,\e}(\xi_2,\eta_2)\|_{\LN}
 +\|h_{2,\e}(\xi_1,\eta_1)-h_{2,\e}(\xi_2,\eta_2)\|_{Y} \\
&\quad \le CM_1e^{CM_1\e}\e \big(\|\xi_1-\xi_2\|_{H^2(\RN)} +\|\eta_1-\eta_2\|_X\big).
\end{align*}

We define a map $\mathbb{L}_0:H_r^2(\RN)\times E_0^r\to L_r^2(\RN)\times Y_r$ by
\[ \mathbb{L}_0(\xi,\eta) =\big(\mathcal{L}_1\xi, ~\mathcal{L}_2\eta\big). \]
Then we can choose constants $M_1\ge 1$ and $\e^*>0$ such that if $0<\e<\e^*$ then the map
$\Gamma_\e:S_1\to S_1$ defined by
\[ \Gamma_\e (\xi,\eta) =\big(\mathcal{L}_1^{-1}h_{1,\e}(\xi,\eta), ~
 \mathcal{L}_2^{-1}h_{2,\e}(\xi,\eta)\big). \]
is a well-defined contraction map. Hence for each $0<\e<\e^*$, there exists
a unique element $(\xi_\e^*,\eta_\e^*)\in S_1$ such that
\[ \mathcal{L}_1\xi_\e^* =h_{1,\e}(\xi_\e^*,\eta_\e^*) \quad\mbox{and}\quad
 \mathcal{L}_2\eta_\e^* =h_{2,\e}(\xi_\e^*,\eta_\e^*). \]
Therefore $(u_1,u_2)$ defined by
\[ \left\{ \begin{aligned} u_1(r) &= -\ln 2 +\e\xi_\e^*(r), \\
 u_2(r) &= W_0(\e r) +2\ln\e -(b/2)\e\xi_\e^*(r) +\e\eta_\e^*(\e r)
 \end{aligned} \right. \]
is a radially symmetric solution of the system \eqref{eq:main}. \\
This completes the proof of Theorem \ref{T121}. ~\hfill $\square$

\noindent{\bf Remark}.
The above argument does not work for the case $b=1$, $N_1=2$, $N_2=0$ and $p_2=-p_1 \neq {\bf 0}$,
which seems to be a subtle case and requires a new approach.


\setcounter{equation}{0}


\end{document}